\documentclass[a4paper,oneside,10pt,final,reqno]{amsart}

\usepackage{amssymb}
\usepackage[mathcal]{euscript}
\usepackage[frame,cmtip,curve,arrow,matrix,line,graph]{xy}
\usepackage{enumerate}

\usepackage[dvipdfmx]{hyperref}

\usepackage[all, warning]{onlyamsmath}

\numberwithin{equation}{subsection}
\allowdisplaybreaks

\theoremstyle{definition}
\newtheorem{thm}{Theorem}[subsection]

\newtheorem{lem}[thm]{Lemma}
\newtheorem{cor}[thm]{Corollary}
\newtheorem{dfn}[thm]{Definition}
\newtheorem{eg}[thm]{Example}
\newtheorem{fct}[thm]{Fact}
\newtheorem{rmk}[thm]{Remark}
\newtheorem*{dfn*}{Definition}
\newtheorem*{lem*}{Lemma}
\newtheorem*{cor*}{Corollary}
\newtheorem*{fct*}{Fact}
\newtheorem*{rmk*}{Remark}

\newcommand{\ve}{\varepsilon}
\newcommand{\wt}{\widetilde}
\newcommand{\oln}{\overline}

\newcommand{\bbA}{\mathord{\mathbb{A}}}
\newcommand{\bbC}{\mathord{\mathbb{C}}}
\newcommand{\bbZ}{\mathord{\mathbb{Z}}}

\newcommand{\g}{\mathord{\mathfrak{g}}}
\newcommand{\C}{\mathord{\mathfrak{C}}}
\newcommand{\R}{\mathord{\mathcal{R}}}
\newcommand{\X}{\mathord{\mathcal{X}}}
\newcommand{\Y}{\mathord{\mathcal{Y}}}
\newcommand{\shA}{\mathord{\mathcal{A}}}
\newcommand{\shD}{\mathord{\mathcal{D}}}
\newcommand{\shK}{\mathord{\mathcal{K}}}
\newcommand{\shL}{\mathord{\mathcal{L}}}
\newcommand{\shO}{\mathord{\mathcal{O}}}
\newcommand{\shP}{\mathord{\mathcal{P}}}
\newcommand{\shQ}{\mathord{\mathcal{Q}}}

\newcommand{\catC}{\mathord{\mathcal{C}}}
\newcommand{\catS}{\mathord{\mathcal{S}}}
\newcommand{\calH}{\mathord{\mathcal{H}}}
\newcommand{\calC}{\mathord{\mathcal{C}}}
\newcommand{\calU}{\mathord{\mathcal{U}}}
\newcommand{\frkh}{\mathord{\mathfrak{h}}}

\newcommand{\frkm}{\mathord{\mathfrak{m}}}
\newcommand{\frkS}{\mathord{\mathfrak{S}}}
\newcommand{\Diff}{\mathord{\mathcal{D}\mathit{iff}}}

\newcommand{\Sh}{\mathord{\mathcal{S}h}}
\newcommand{\CSh}{\mathord{\mathcal{CS}h}}
\newcommand{\DSh}{\mathord{\mathcal{DS}h}}
\newcommand{\M}{\mathord{\mathcal{M}}}
\newcommand{\CM}{\mathord{\mathcal{CM}}}

\newcommand{\CFM}{\mathord{\mathcal{CFM}}}
\newcommand{\MO}{\mathord{\mathcal{M}_{\shO}}}
\newcommand{\CMO}{\mathord{\mathcal{CM}_{\shO}}}
\newcommand{\DMO}{\mathord{\mathcal{DM}_{\shO}}}

\newcommand{\D}{\mathord{\mathcal{D}}}

\newcommand{\Ho}{\mathord{\mathcal{H}\mathit{o}}}
\newcommand{\Cocom}{\mathord{\mathcal{C}\mathit{ocom}}}

\newcommand{\opBV}{\mathord{\mathcal{BV}}}
\newcommand{\opEnd}{\mathord{\mathcal{E}\mathit{nd}}}
\newcommand{\opLie}{\mathord{\mathcal{L}\mathit{ie}}}

\newcommand{\adm}{\text{adm}}

\newcommand{\ch}{\text{ch}}
\newcommand{\DR}{\text{DR}}

\newcommand{\gr}{\mathop{\mathrm{gr}}\nolimits}
\newcommand{\id}{\mathop{\mathrm{id}}\nolimits}
\newcommand{\sgn}{\mathop{\mathrm{sgn}}\nolimits}
\newcommand{\spn}{\mathop{\mathrm{span}}\nolimits}

\newcommand{\MC}{\mathop{\mathrm{MC}}\nolimits}
\newcommand{\Aut}{\mathop{\mathrm{Aut}}\nolimits}
\newcommand{\Der}{\mathop{\mathrm{Der}}\nolimits}
\newcommand{\Hom}{\mathop{\mathrm{Hom}}\nolimits}
\newcommand{\Lie}{\mathop{\mathrm{Lie}}\nolimits}
\newcommand{\Sym}{\mathop{\mathrm{Sym}}\nolimits}
\newcommand{\Cone}{\mathop{\mathrm{Cone}}\nolimits}
\newcommand{\Spec}{\mathop{\mathrm{Spec}}\nolimits}
\newcommand{\Spf}{\mathop{\mathrm{Spf}}\nolimits}

\newcommand{\tprd}{{\textstyle{\prod}}}
\newcommand{\tsum}{{\textstyle{\sum}}}

\newcommand{\inj}{\hookrightarrow}
\newcommand{\surj}{\twoheadrightarrow}
\newcommand{\simto}{\xrightarrow{\sim}}
\newcommand{\longto}{\longrightarrow}
\newcommand{\longsimto}{\xrightarrow{\ \sim \ }}
\newcommand{\longinj}{\lhook\joinrel\longrightarrow}
\newcommand{\longsurj}{\relbar\joinrel\twoheadrightarrow}

\makeatletter
\def\leftharpoondownfill@{\arrowfill@\leftharpoondown\relbar\relbar}
\def\rightharpoonupfill@{\arrowfill@\relbar\relbar\rightharpoonup}
\newcommand{\xrightleftharpoons}[2][]{\mathrel{%
  \raise.22ex\hbox{$\ext@arrow 3095\rightharpoonupfill@{\phantom{#1}}{#2}$}%
  \setbox0=\hbox{$\ext@arrow 0359\leftharpoondownfill@{#1}{\phantom{#2}}$}%
  \kern-\wd0 \lower.22ex\box0}}
\makeatother

\begin{document}

%%%%%%%%%%%%%%%%%%%% Authors & Paper data %%%%%%%%%%%%%%%%%%%%

\title{Jacobi complexes on the Ran space}

\author{Shintarou Yanagida}
\address{Graduate School of Mathematics, Nagoya University 
Furocho, Chikusaku, Nagoya, Japan, 464-8602.}
\email{yanagida@math.nagoya-u.ac.jp}

%\subjclass[2010]{05A15}
%\date{July 12 --, 2016}
\date{August 26, 2016}

%%%%% front matter %%%%%

\begin{abstract}
We study local theory of moduli schemes 
using the framework of the Ran space.
With the help of the study of sheaves and complexes over the Ran space 
by Beilinson and Drinfeld in their theory of chiral algebras,
we revisit Ran's works on the Jacobi complexes 
(the Chevalley complexes for sheaves of Lie algebras on the Ran space),
the universal deformation rings of moduli problems,
the higher Kodaira-Spencer maps,
and construction of Hitchin-type flat connections.
We give rigorous treatments in the algebraic setting,
which seems to be new.
\end{abstract}

\maketitle
\tableofcontents

%%%%%%%%%%%%%%%%%%%%%%%%%%%%%%%%%%%%%%%%%%%%%%%%%%%%%%%%%%%%%%%%%%%%%%
%%%%%%%%%%%%%%%%%%%%%%%%%%%%%%%%%%%%%%%%%%%%%%%%%%%%%%%%%%%%%%%%%%%%%%
\setcounter{section}{-1}
\section{Introduction}
\label{sect:intro}

The purpose of this note is to understand the works 
\cite{R0,R1,R2,R3} by Z.~Ran on local geometry of moduli spaces.
His study is based on sheaves of dg Lie algebras 
on what he called the very symmetric space,
which is now called the Ran space.

Let us briefly explain the notion of Ran space. 
The motivation of its definition comes from deformation theory, 
as explained in \cite{R0,R1},
Study on local geometry of moduli space is equivalent to 
deformation theory of the data considered.
Ran's idea is that 
the $n$-th order deformations of objects related to $X$ should be 
controlled by some sheaves on the space parametrising $n$-tuples 
of points on $X$.
The latter space is the ($n$-th filtration of) Ran space $\R(X)_n$.
%Since the Kodaira-Spencer theory 
%of deformation of complex manifolds \cite{KS} ,
%deformation theory has been one of the central sources 
%of new ideas and notions in modern mathematics.
Precisely speaking,
%Let us give a %In \cite{R0} Ran started his study of moduli spaces 
%in terms of the Ran space.
for a given space or scheme $X$, the Ran space $\R(X)$ 
is the topological space of finite points in $X$.
Thus a point in $\R(X)$ is a finite tuple $\{x_1,\ldots,x_n\} \subset X$,
where $x_i$'s may coincide.
It has the natural filtration $\R(X)_1 \subset \R(X)_2 \subset \cdots$ 
with $\R(X)_n$ the subspace of $n$-tuples of points .

The  Ran space looks too simple at the first glance.
Actually it is surprisingly strong 
as the results in Ran's works  show.
Let us explain a few of them.
In \cite{R1}, the universal deformation ring is constructed 
in terms of what he called the Jacobi complex.
The Jacobi complex is the Chevalley complex of the dg Lie algebra on $\R(X)$, 
as we will see later in \S\ref{sect:jac}.
In \cite{R2}, a general framework to construct a flat connection 
related to  moduli spaces is presented.
It includes the Hitchin connection \cite{H} 
on the space of generalized theta functions.

The first purpose of this note is to 
understand Ran's work in the algebraic setting.
Basically Ran worked over a Hausdorff space $X$,
Since for a scheme $X$ the Ran space $\R(X)$ 
is not a scheme (nor an ind-scheme),
one should take care to treat sheaves and complexes on $\R(X)$.

Another remark is that large part of Ran's construction utilizes 
a resolution of sheaves on $\R(X)$,
so that it is natural to restate his results in a derived or homotopy setting.
This is our second purpose.
 
Fortunately, at present we have a reference of sheaves on $\R(X)$.
In \cite{BD} Beilinson and Drinfeld built the theory of chiral algebras,
and they utilized the category of sheaves (and $\shD$-modules) on $\R(X)$.
We will fully use their treatment. 

Let us say a few words on chiral algebras.
They are some Lie objects in a non-standard tensor category of 
$\shD$-modules on a given curve.
The theory of chiral algebras is almost (but not exactly) equivalent to 
that of vertex algebras, 
which encodes the local symmetry of conformal field theory.
In \cite{BD} the Ran space is used in 
the equivalence of chiral algebras and factorization algebras.
The latter is a sheaf on $\R(X)$ with additional conditions.

Let us explain the organization of this note.
In \S\ref{sect:Ran} we explain the Ran space and sheaves on it 
following \cite{BD}.
Although the theory of chiral algebras use $\shD$-modules on $\R(X)$ 
extensively,
our study requires only $\shO$-modules so that our presentation is 
a restricted version of \cite{BD}.
In \S\ref{sect:jac} we give the definition of Jacobi complex
after the preparation of general study of Chevalley complex of Lie algebra 
objects in arbitrary tensor category.
As  an application of Jacobi complex, 
we construct the universal deformation ring 
and the higher Kodaira-Spencer maps for deformation of schemes
in \S\ref{sect:hks:R}.
A modified version for the moduli space of $G$-bundles 
is given in \S\ref{sect:hks:G-bundle} 
In \S\ref{sect:BV} we construct flat connections on 
the homology of Jacobi complex.
The Hitchin connection is the basic  example of this construction.

As we will see in the main text, our study of local geometry of moduli spaces 
has many common features with \cite{BD}.
This is not so surprising, since 
conformal vertex algebras are intimately connected to 
moduli problems on algebraic curves, as explained in \cite[Chap.\ 16--17]{FBZ}.
%Beginning with the work of Tsuchiya, Ueno and Yamada \cite{TUY} 
%formulating the WZW model on the moduli space of pointed complex curves,
%there are now plenty amount of literature concerning this topic.
Beyond such a conceptual explanation, we expect a kind of functorial 
correspondence from our study to the theory of chiral algebras.
For example, our construction of flat connections in \S\ref{sect:BV} 
goes along almost the same line as the construction of flat connections on 
chiral homology in \cite{BD}.
In the case of Hitchin connection, the expected correspondence 
should be the so-called Verlinde isomorphism.
We will study this problem in future.

%%%%%%%%%%%%%%%%%%%%%%%%%%%%%%%%%%%%%%%%%%%%%%%%%%%%%%%%%%%%%%%%%%%%%%
\subsection*{Notation}

For a set $S$, $|S|$ denotes its cardinality.

For a category $\catC$, the notation $A \in \catC$ 
means that $A$ is an object of $\catC$.
The word ``tensor category'' is used 
in the meaning of a category $\catC$ with multiplication $\otimes_{\catC}$ 
which is symmetric with the commutator 
$s_{A,B}: A \otimes_{\catC} B \to B \otimes_{\catC} A$
in the sense of Mac Lane.
The word ``symmetric monoidal category'' is used 
in the meaning of tensor category with unit.
The phrase ``dg'' means ``differential graded" as usual.

We will work on a fixed field $k$.
The symbol $\otimes$ denotes the standard tensor product 
in the category of $k$-vector spaces unless otherwise stated.

%$\catC(k)$ denotes the category of 
%dg $k$-vector spaces (or complexes).
The grading of a dg $k$-vector space (namely a complex) $C$ 
is denoted by the superscript like $C^p$,
and the differential is assumed to be of degree $+1$.
The $n$-shift of $C$ is dented by $C[n]$
in the meaning of $C[n]^p= C^{n+p}$.

We also use the language of operads freely.
%and the operad $\opBV$ of Batalin-Vilkovisky algebras.
Let us name \cite{LV} for a reference out of plenty numbers of literature.
$\opLie$ denotes the operad of Lie algebras. 
%For an operad $\opB$ and a $k$-linear tensor category $\M$,
%denote by $\opB(\M)$ the category of $\opB$-algebras in $\M$.

For a sheaf $F$ on a topological space, 
the notation $t \in F$ means that $t$ is a local section of $F$.
For a complex $F^\bullet$ of sheaves, 
$\calH^i(F^\bullet)$ denotes the $i$-th cohomology sheaf.

Finally, $\frkS_n$ denotes the $n$-th symmetric group.

%%%%%%%%%%%%%%%%%%%%%%%%%%%%%%%%%%%%%%%%%%%%%%%%%%%%%%%%%%%%%%%%%%%%%%
\subsection*{Acknowledgements}

The author is supported by the Grant-in-aid for 
Scientific Research (No.\ 16K17570), JSPS.
This work is also supported by the 
JSPS for Advancing Strategic International Networks to 
Accelerate the Circulation of Talented Researchers
``Mathematical Science of Symmetry, Topology and Moduli, 
  Evolution of International Research Network based on OCAMI''.

This note is written during the author's stay at UC Davis 
in the summer 2016. 
The author would like to thank the institute for support and hospitality.
%and Professor M.~Mulase for the discussion around opers.

%%%%%%%%%%%%%%%%%%%%%%%%%%%%%%%%%%%%%%%%%%%%%%%%%%%%%%%%%%%%%%%%%%%%%%
%%%%%%%%%%%%%%%%%%%%%%%%%%%%%%%%%%%%%%%%%%%%%%%%%%%%%%%%%%%%%%%%%%%%%%
\section{Ran space and sheaves on it}
\label{sect:Ran}

Following \cite[\S2.1]{R1} and \cite[\S3.4, \S4.2]{BD},
we recall some basic notions on Ran space and sheaves on it.

%%%%%%%%%%%%%%%%%%%%%%%%%%%%%%%%%%%%%%%%%%%%%%%%%%%%%%%%%%%%%%%%%%%%%%
\subsection{Ran space}

For a topological space $X$, denote by $\R(X)$ the \emph{Ran space}
which is the set of all non-empty finite subsets in $X$
with the strongest topology such that 
the obvious map 
\[
 r_I: \, X^I \longto \R(X)
\]
is continuous for any finite index set $I$.
The point of $\R(X)$ associated to a finite subset $S \subset X$ 
is denoted by $[S]$.

For $n \in \bbZ$,
denote by $\R(X)_n$ the subspace of $\R(X)$ consisting of $[S]$ 
such that  $|S| \le n$.
We have a projection map 
\[
 r_n: \, X^n \longsurj \R(X)_n = X^n/\sim,
\]
where $(x_i)_{i=1}^n \sim (x_i')_{i=1}^n$ if and only if 
$\{x_i\}=\{x_i'\}$.
The map $r_n$ is nothing but $r_I$ with $I = \{1,\ldots,n\}$.
%, and is an open map.

We have an increasing filtration 
\[
 \R(X)_0 = \emptyset \, \subset \, \R(X)_1 = X \, \subset \, 
 \R(X)_2 = \Sym^2(X) \, \subset \, \R(X)_3 \, \subset \, 
 \cdots \, \subset \,
 \R(X)_\infty := \R(X).
\]
Here $\Sym^n(X):= X^n/\frkS_n$ is the usual symmetric product.

Let us set $\R(X)^\circ_n := \R(X)_n \setminus \R(X)_{n-1}$.
It coincides with the complement of the partial diagonals in $\Sym^n(X)$.
Thus
\begin{equation}\label{eq:Un}
 \R(X)^\circ_n = U^{(n)}/\frkS_n, \quad
 U^{(n)} := X^n \setminus \cup \, (\text{partial diagonals}).
\end{equation}
$\R(X)^\circ_n$ is nothing but the configuration space of $n$ points in $X$.

For any surjection $\pi: J \twoheadrightarrow I$ 
denote by 
\begin{equation}\label{eq:Delta^pi}
 \Delta^{(\pi)} \equiv \Delta^{(J/I)}: \, 
 X^I \longinj X^J, \quad
 (x_i)_{i \in I} \longmapsto (y_j:=x_{\pi(j)})_{j \in J}.
\end{equation}
the diagonal embedding.
Then we have $r_J \Delta^{(J/I)} = r_I$,
and $\R(X)$ is the inductive limit of the spaces $X^I$ 
with respect to these embeddings $\Delta^{(J/I)}$.

$\R(X)$ is a commutative semigroup under the continuous map 
\begin{equation}\label{eq:u}
 u: \, \R(X) \times \R(X) \longto \R(X),\quad
 \left([S],[T]\right) \longmapsto [S \cup T].
\end{equation}
It is the direct limit of %the maps
$u_{m,n}: \R(X)_m \times \R(X)_n \to \R(X)_{m+n}$
given by the same operation.
We have the relation
\begin{equation}\label{eq:r=urr}
 r_{m+n} = u_{m,n} \circ (r_m \times r_n).
\end{equation}
We also have a continuous map
\begin{equation}\label{eq:v}
 v_m^n: \, \R(\R(X)_m)_n \longto \R(X)_{m n},\quad
 \left[\{[S_1],\ldots,[S_n]\}\right] 
 \longmapsto [S_1 \cup \cdots \cup S_n].
\end{equation}

%%%%%%%%%%%%%%%%%%%%%%%%%%%%%%%%%%%%%%%%%%%%%%%%%%%%%%%%%%%%%%%%%%%%%%
\subsection{Sheaves on Ran space}
\label{subsec:Ran:XS}

Hereafter let $X$ be a scheme over a field $k$ 
with finite cohomological dimension.
Sheaves on schemes mean the ones in the \'{e}tale topology.

Let us introduce some notations for the sheaves on $X$,
which will be used throughout this note. 
Denote by $\Sh(X)$ the category of sheaves of $k$-vector spaces on $X$.
An $\shO$-module on $X$ means a 
quasi-coherent sheaf of $\shO_X$-modules over $X$.
$\MO(X)$ denotes the category of $\shO$-modules.
$\CMO(X)$ denotes the dg category of complexes of $\shO$-modules on $X$,
and $\DMO(X)$ denotes the corresponding derived category, namely 
the localization of $\CMO(X)$ by quasi-isomorphisms.
For a morphism $f:X \to Y$ of $k$-schemes,
$f^*$ and $f_*$ denotes the usual pull-back and push-forward functors 
on $\shO$-modules.

Let $\catS$ be the category of finite non-empty sets and surjections.
Following \cite[\S3.4.1]{BD} we introduce the notion of sheaves 
on the Ran space $\R(X)$.

\begin{dfn}\label{dfn:Ran:O-mod}
An \emph{$\shO$-module on $\R(X)$} 
is a rule $F$ assigning to each $I \in \catS$ an $\shO_{X^I}$-module  $F_I$  
and to each  $\pi: J \twoheadrightarrow I$ in $\catS$ 
an isomorphism
\[
 \nu^{(\pi)}_F: \, \Delta^{(\pi) *}F_J \longsimto F_I
\]
of $\shO_{X^I}$-modules 
compatible with the composition of surjections,
namely for any $\rho:K \twoheadrightarrow J$ and $\pi:J \surj I$ we have 
\[
 \nu_F^{(\pi)} \circ \Delta^{(\pi) *} (\nu_F^{(\rho)}) 
 = \nu_F^{(\pi \circ \rho)},
\]
and also have $\nu_F^{(\id)} = \id_{F_I}$.
%Denote by $\Sh(\R(X))$ the category of sheaves of 
%$k$-vector spaces on $\R(X)$. 
Denote by $\MO(\R(X))$ the category of $\shO$-modules on $\R(X)$. 
\end{dfn}

We have a similar definition of a sheaf of $k$-vector spaces on $\R(X)$
by replacing $\Delta^{(\pi) *}$ with $(\Delta^{(\pi)})^{-1}$.
Denote by $\Sh(\R(X))$ the corresponding category.

We want to consider a derived category for $\MO(\R(X))$, 
but since %$\R(X)$ is not a scheme nor even an ind-scheme,
this category is only exact in the sense of Quillen and not abelian, 
we need to take some detour to handle complexes of sheaves on $\R(X)$.
Following \cite[\S4.2]{BD},
let us consider a larger category of sheaves living on an enlarged 
`space' $X^{\catS}$ above $\R(X)$.

%As before denote by $\catS$ 
%the category of finite non-empty sets and surjections.
Let $X^{\catS}$ be the diagram of schemes on the opposite category 
$\catS^\circ$ given by 
\[
 X^{\catS}:\, 
 I \longmapsto X^I, \quad
 (\pi:J \longsurj I) \longmapsto 
 (\Delta^{(\pi)}: X^I \longinj X^J).
\]
Here $\Delta^{(\pi)}$ is the diagonal map given in \eqref{eq:Delta^pi}.
Note that it is a closed embedding.

\begin{dfn}\label{dfn:!-O-mod}
A \emph{$!$-$\shO$-module on $X^{\catS}$} is 
a rule $F$ assigning to each $I \in \catS$ an $\shO_{X^I}$-module $F_I$ 
and to each $\pi:J \twoheadrightarrow I$ a morphism of $\shO_{X^J}$-modules 
\[
 \theta_F^{(\pi)}: \,  F_J \longto \Delta^{(\pi)}_* F_I
\]
compatible with the compositions of surjections. 
%and $\theta_F^{(\id)} = \id_{F_I}$.
%
Denote by $\MO(X^{\catS})$ the category of $!$-$\shO$-modules on $X^{\catS}$.
\end{dfn}

$\MO(X^{\catS})$ is an abelian $k$-linear category,
and the corresponding dg category of complexes 
and the derived category are denoted by 
$\CMO(X^{\catS})$ and $\DMO(X^{\catS})$ respectively.
Similarly we can define a $!$-sheaf of $k$-vector spaces on $X^{\catS}$.
Denote by $\Sh(X^{\catS})$ the corresponding cage tory.
It is abelian and $k$-linear,
and we denote by $\CSh(X^{\catS})$ and $\DSh(X^{\catS})$  
the associated categories.

Now we want to consider a subcategory of $\CMO(X^{\catS})$ 
formed by complexes on $\R(X)$.
Note that $\MO(\R(X))$ is naturally a full subcategory of $\MO(X^{\catS})$ 
by adjunction.

\begin{dfn}\label{dfn:adm}
A complex $F \in \CMO(X^{\catS})$ is called \emph{admissible} 
if for each $\pi: J \twoheadrightarrow I$ in $\catS$ 
the morphism $\theta_F^{(\pi)}$ yields a quasi-isomorphism 
$R\Delta^{(\pi) !} F_J \simto  F_I$.
The category of admissible complexes is denoted by 
$\CMO(X^{\catS})_{\adm}$
\end{dfn}

$\CMO(X^{\catS})_{\adm}$ is a full dg subcategory of $\CMO(X^{\catS})$ 
closed under quasi-isomorphisms,
and yields a full triangulated subcategory 
$\DMO(X^{\catS})_{\adm} \subset \DMO(X^{\catS})$.
%Note that 
%$\MO(\R(X))$ is a full subcategory of $\CMO(X^{\catS})_{\adm}$
%under the embeddings
%$\MO(\R(X)) \subset \MO(X^{\catS})\subset \CMO(X^{\catS})$.
%
We have a  similar discussion for  the category $\CSh(X^{\catS})$ 
of complexes of $!$-sheaves of $k$-vector spaces.
Thus we have subcategories $\CSh(X^{\catS})_{\adm}$ and 
$\DSh(X^{\catS})_{\adm}$.

Now we can state 

\begin{dfn*}
The \emph{derived category $\DMO(\R(X))$ of $\shO$-modules}
(resp.\ 
the \emph{derived category $\DSh(\R(X))$ of sheaves of $k$-vector spaces})  
on $\R(X)$ is defined respectively to be
\[
 \DMO(\R(X)) := \DMO(X^{\catS})_{\adm},\quad
 \DSh(\R(X)) := \DSh(X^{\catS})_{\adm}.
\]
\end{dfn*}

%Now we turn to a filtration on $!$-sheaves.
%
%\begin{dfn*}
%Each $F \in \MO(X^{\catS})$ has a canonical filtration
%\[
% F_1 \subset F_2 \subset \cdots \subset F_n \subset \cdots \subset F
%\]
%where $F_n$ is the smallest subsheaf such that $(F_n)_I = F_I$ 
%for $|I| \le n$.
%It is called the \emph{Cousin filtration}.
%\end{dfn*}

For $F \in \MO(X^{\catS})$, 
we have an $\catS^\circ$-diagram of $k$-vector spaces 
$I \mapsto \Gamma(X^I,F_I)$.
Then we can define 
\begin{equation}\label{eq:Gamma(X^S)}
 \Gamma(X^{\catS},F) := \varinjlim_I \Gamma(X^I,F_I).
\end{equation}
%This $k$-vector space has a filtration 
%$\Gamma(X^{\catS},F)_{\bullet}$
%induced by the Cousin filtration $F_\bullet$ of $F$.
%We also call it the \emph{Cousin filtration} on $\Gamma(X^{\catS},F)$.

By the discussion in \cite[\S4.2.2]{BD},
we have the derived functor 
\[
 R\Gamma(X^{\catS},-): \, \DMO(X^{\catS}) \longto \D(k),
\]
where $\D(k)$ is the derived category of $k$-vector spaces.
%$R\Gamma(X^{\catS},-)$ is indeed an exact functor 
%\[
% R\Gamma(X^{\catS},-): \, \DMO(X^{\catS}) \longto \DF(k),
%\]
%where $\DF(k)$ is the filtered derived category of $k$-vector spaces
%(see Notations in \S\ref{sect:intro} for the filtered derived category).
%Thus an object $V \in D F(k)$ is a complex equipped with an increasing 
%filtration $V_{\bullet}$ such that $V_n=0$ for $n\ll 0$ 
%and $\cup_n V_n = V$.
%Morphisms should preserve filtrations 
%localized by filtered quasi-isomorphisms.

For $F \in \DMO(X^{\catS})$, we denote 
\[
 H^{\bullet}(X^{\catS},F) := H^{\bullet} R\Gamma(X^{\catS},F).
\]
%The Cousin filtration on \eqref{eq:Gamma(X^S)} yields a spectral sequence 
%\[
% E^{p,q}_r \Longrightarrow H^{\bullet}(X^{\catS},F).
%\]
%It is called the \emph{Cousin spectral sequence}.

%Let us now cite \cite[\S4.2.3 Lemma]{BD}:
%For $F \in \DMO(X^{\catS})_{\adm}$, we have
%\begin{equation}\label{eq:4.2.3}
% \gr_n R\Gamma(X^{\catS},F) \simto R\Gamma(\R(X)^{\circ}_n,F^\circ_n).
%\end{equation}
%Here the sheaf $F^{\circ}_n$ on $\R(X)^{\circ}_n$ is defined to be 
%the descent of the $\frkS_n$-equivalent sheaf $F_n$ on 
%$U^{(n)} \subset X^n$ to $\R(X)^{\circ}_n = U^{(n)}/\frkS_n$ 
%(see \eqref{eq:Un}).
%As a corollary, for $F \in \DMO(X^{\catS})_{\adm}$ we have
%$E_1^{p,q}=H^{p+q}(\R(X)^{\circ}_{-p},F^\circ_{-p})$.

\begin{rmk}\label{rmk:X=hausdorff}
As explained in \cite[\S4.2.4]{BD}, 
in the situation where $X$ is a locally compact Hausdorff space,
$R(X)$ is the inductive limit of the diagram $X^{\catS}$
and $\Sh(\R(X))$ is an abelian category 
thus one can define $\DSh(\R(X))$ directly.
In this case one has an equivalence 
$\DSh(X^{\catS})_{\adm} \simeq \DSh(\R(X))$.
%
%The map $r_I:X^I \to \R(X)$ is finite, so 
%$r_{I *}: \Sh(X^I) \to \Sh(\R(X))$ is exact.
%It induces a right exact functor 
%\[
% r_*:\Sh^{!}(X^{\catS}) \longto \Sh(\R(X)),\quad
% F \longmapsto \varinjlim_I r_{I *}F_I.
%\]
%It has a right adjoint $r^!$.
\end{rmk}

The restriction of the functor $R\Gamma(X^{\catS},-)$ to $\DMO(\R(X))$ 
is denoted by $R\Gamma(\R(X),-)$.
%It has the Cousin filtration $R\Gamma(\R(X),F)_n = R\Gamma(\R(X)_n,F)$.
%We have
%$\gr_n R\Gamma(\R(X),F) \simeq R\Gamma(\R(X)^\circ_n,F^\circ_n)$
%by \eqref{eq:4.2.3}.

%%%%%%%%%%%%%%%%%%%%%%%%%%%%%%%%%%%%%%%%%%%%%%%%%%%%%%%%%%%%%%%%%%%%%%
\subsection{Convolution tensor product}

We turn to a tensor structure on the categories of sheaves on 
$\R(X)$ and on $X^{\catS}$ 
following \cite[\S3.4.10]{BD}.
As in the previous subsection, 
$X$ is a $k$-scheme of finite cohomological dimension.

\begin{dfn}\label{dfn:conv}
For $F,G \in \MO(X^{\catS})$, we define the \emph{convolution product}
$F \otimes^* G \in \MO(X^{\catS})$ by 
\[
 (F \otimes^* G)_{I} := 
 \oplus_{\pi:I \surj \{1,2\}} F_{\pi^{-1}(1)} \boxtimes G_{\pi^{-1}(2)}.
\]
The structure morphisms $\theta^{(\pi)}$ are defined naturally.
The resulting tensor category is denoted by $\MO(X^{\catS})^*$.
It induces a tensor structure on the full subcategory $\MO(\R(X))$,
and the resulting tensor category is denoted by $\MO(\R(X))^*$.
\end{dfn}

Let us compare this tensor structure with the natural one 
on the category $\MO(X)$ of $\shO_X$-modules,
namely the tensor product $\otimes_{\shO_X}$.
Denote by $\MO(X)^*$ the corresponding tensor category.

We have a projection functor 
\[
 \MO(X^{\catS}) \longto \MO(X),\quad
 F \longmapsto F_{\{1\}},
\]
where $\MO(X)$ is the category of $\shO_X$-modules.
It has a left adjoint 
\begin{equation}\label{eq:Delta*S}
 \Delta^{(\catS)}_*: \, \MO(X) \longto \MO(X^{\catS})
\end{equation}
given by 
\[
 (\Delta^{(\catS)}_*M)_I := \Delta_*^{(I)} M, \quad
 \theta^{(\pi:\, J \surj I)} := \id_{\Delta^{(J)}_* M}.
\]
Here 
$\Delta^{(I)} := \Delta^{(I \surj \{1\})}: X \hookrightarrow X^I$ 
is the diagonal map.
The functor $\Delta^{(\catS)}_*$ is fully faithful.

The functor $\Delta^{(\catS)}_*$ induces 
$\CMO(X) \longto \CMO(X^{\catS})$,
and in fact the admissibility in Definition \ref{dfn:adm} is 
automatically satisfied.
Thus we have
\[
 \Delta^{(\catS)}_*: \, \CMO(X) \longto \CMO(\R(X)),\quad
 \DMO(X) \longto \DMO(\R(X))
\]

%The embedding $\Delta^{(\catS)}_*$ 
%given in \eqref{eq:Delta*S} is compatible with the tensor product 
%$\otimes_{\shO_X}$ and the convolution product $\otimes^*$,
%so we have 
Now we immediately have 

\begin{lem}\label{lem:Delta*S}
$\Delta^{(\catS)}_*$ is compatible with the tensor product 
$\otimes_{\shO_X}$ and the convolution product $\otimes^*$.
Thus it yields a fully faithful embedding of tensor category
\[
 \Delta^{(\catS)}_*: \, \MO(X)^* \longinj \MO(X^{\catS})^*,
 \quad 
 M \longmapsto 
 F=(F_I := \Delta_*^{(I)} M, \, 
    \theta_F^{(\pi:J \surj I)} := \id_{\Delta^{(J)}_* M}).
\]
\end{lem}

By \cite[\S4.2.5]{BD}, if $X$ is quasi-compact,
then the functors $\Gamma(X^{\catS},-)$ and $R\Gamma(X^{\catS},-)$  
are tensor functors with respect to $\otimes^*$.
So are the functors $\Gamma(\R(X),-)$ and $R\Gamma(\R(X),-)$  

\begin{rmk}\label{rmk:conv}
Continuing Remark \ref{rmk:X=hausdorff},
if $X$ is Hausdorff, then %the tensor product 
$\otimes^*$ on $\MO(\R(X))$ can be written as 
\begin{equation}\label{eq:rmk:conv}
 F \otimes^* G := u_* (F \boxtimes G),
\end{equation}
where $u:\R(X) \times \R(X) \to \R(X)$ 
is the commutative semigroup structure given in \eqref{eq:u}.
This convolution product is used in \cite{R1} for the construction 
of the Jacobi complex and its OS-structure.
%Since $u$ is a finite map, $\otimes^*$ is an exact functor.
%If $X$ is quasi-compact, then we have identifications
%\[
% \Gamma(\R(X),F \otimes^* G) = 
% \Gamma(\R(X),F) \otimes \Gamma(\R(X),G),\quad
% R\Gamma(\R(X),F \otimes^* G) = 
% R\Gamma(\R(X),F) \otimes R\Gamma(\R(X),G),
%\] 
%so that $\Gamma(\R(X),-)$ and $R\Gamma(\R(X),-)$ are tensor functors.
\end{rmk}

%%%%%%%%%%%%%%%%%%%%%%%%%%%%%%%%%%%%%%%%%%%%%%%%%%%%%%%%%%%%%%%%%%%%%%
%%%%%%%%%%%%%%%%%%%%%%%%%%%%%%%%%%%%%%%%%%%%%%%%%%%%%%%%%%%%%%%%%%%%%%
\section{Jacobi complex}
\label{sect:jac}

Let $k$ be a field of characteristic $0$.
%and $S,X$ be schemes over $k$.

%%%%%%%%%%%%%%%%%%%%%%%%%%%%%%%%%%%%%%%%%%%%%%%%%%%%%%%%%%%%%%%%%%%%%%
\subsection{Chevalley complex}
\label{subsec:jac:chv}

%The presentation in the previous subsections is a little abstract one,
%so let us give a down-to-earth description.
Let us recall the Chevalley complex of Lie algebra.
Since we will study Lie algebra objects in various categories,
let us spell out in a general form.
So let $\M$ be an abelian $k$-linear symmetric monoidal category 
with the tensor product $\otimes_{\M}$.
We will omit the symmetrizer $M \otimes_{\M} N \simto N \otimes_{\M} M$
in the following presentation.

Denote by $\CM$ the dg category of complexes of objects in $\M$.
For a complex  $V=(V^\bullet,d) \in \CM$,
denote by $T^n(V)$ the $n$-th tensor power 
\[
 T^n(V) := V \otimes_{\M} \cdots \, \otimes_{\M} V.
\]
The dg tensor category $\CM$ %of complexes of $\shO_S$-modules 
has the commutator 
\begin{equation}\label{eq:commutator}
 R_{V,W}: V \otimes_{\M} W \longsimto W \otimes_{\M} V,\quad 
 v \otimes w \longmapsto (-1)^{|v| \cdot |w|} w \otimes v.
\end{equation}
Here $|v|$ denotes the grading of $v \in V$.
Under this tensor structure 
the symmetric group $\frkS_n$ acts on $T^n(V)$.
Let $\Sym^n(V)$ and $\wedge^n(V)$ be 
the spaces of coinvariants and the co-anti-invariants 
with respect to this $\frkS_n$-action.
Thus we have canonical projections denoted as 
\begin{align*}
&T^n(V) \longsurj S^n(V), \quad 
 v_1 \otimes \cdots \otimes v_n \longmapsto v_1 \cdot \cdots \cdot v_n,
\\
&T^n(V) \longsurj \wedge^n(V), \quad 
 v_1 \otimes \cdots \otimes v_n \longmapsto v_1 \wedge \cdots \wedge v_n.
\end{align*}

If the tensor category $\M$ has a unit $1_{\M}$, then 
there is a natural unital commutative dg $\shO_S$-algebra structure on 
the direct sum 
\[
 \Sym(V) := \oplus_{n\ge0} \Sym^n(V), \quad \Sym^0(V) := 1_{\M}.
\]
It has also the coproduct
$\Delta:\Sym(V) \to \Sym(V) \otimes_{\M} \Sym(V)$
determined by
\begin{enumerate}
\item 
$\Delta(v)=v \otimes 1 + 1 \otimes v$ for $v \in \Sym^1(V)=V$,
\item
$\Delta$ is a morphism of dg algebras in $\M$,
where on $\Sym(V) \otimes_{\M} \Sym(V)$ the algebra structure 
is defined using the commutator \eqref{eq:commutator}.
\end{enumerate} 
Then together with the canonical projection
$\ve: \Sym(V) \to 1_{\M} = \Sym^0(V)$ as the counit,
$\Sym(V)$ is a commutative and cocommutative dg Hopf algebra in $\M$.
It has an increasing filtration 
\begin{equation}\label{eq:S:filtr}
 \Sym(V)_{0} \subset \Sym(V)_{1} \subset \cdots 
 \subset \Sym(V),\quad 
 \Sym(V)_n := \oplus_{i=0}^n \Sym^i(V). 
\end{equation}

\begin{rmk}\label{rmk:chv:copro-filtr}
Note that this filtration respects the augmentation structure 
on $S:=\Sym^i(V)$.
Namely, denoting the augmentation ideal by 
$S^+ := \ker(\ve) = \oplus_{i \ge 1} \Sym^i(V)$ 
and by $\pi:S \surj S^+$ the projection,
we have
$S_{n-1} = \ker(S \to S^{\otimes n} \to (S^+)^{\otimes n})$,
where the first arrow is the $n$-th composition $\Delta^{(n)}$ 
of the coproduct, and the second one is $\pi^{\otimes n}$.
\end{rmk}

Hereafter in this subsection we assume $\M$ is monoidal 
and consider the cocommutative dg coalgebra 
\begin{equation}\label{eq:C(V)}
 C'(V)  := \left(\Sym(V[1]),d',\Delta\right).
\end{equation}
%The $n$-shift of complex is denoted by $[n]$ in the sense 
%\[
% V[n]^p = V^{p+n}.
%\]
%Then we have 
Recall that the shift $[n]$ of complexes  yield
the canonical isomorphisms
\begin{equation}\label{eq:shift-ism}
 V[m] \otimes_{\M} W[n] \longsimto (V \otimes_{\M} W)[m+n],\quad
 v \otimes w \longmapsto (-1)^{p n} v \otimes w
\end{equation}
with $v \in V^p$.
The isomorphisms \eqref{eq:shift-ism} induce 
$t_n: T^n(V[1]) \simto T^n(V) [n]$ 
with $t_n \circ \sigma = \sgn(\sigma) \sigma \circ t_n$ 
for any $\sigma \in \frkS_n$.
Thus we have a canonical isomorphism
\begin{equation}\label{eq:decalage}
 \Sym^n(V[1]) \longsimto \wedge^n(V)[n],\quad
 v_1 \cdot \cdots \cdot v_n \longmapsto 
 (-1)^{\sum_i (n-i) p_i} v_1 \wedge \cdots \wedge v_n,
\end{equation}
where $p_i$ is given by  $v_i \in V^{p_i}$.

Let $L=(L^\bullet,d_{L},[,])$ be a dg Lie algebra in $\M$. 
Thus $(L^\bullet,d_{L}) \in \CM$ 
and the Lie bracket $[,]: L \otimes_{\M} L \to L$ 
is a graded morphism which should satisfy 
\begin{enumerate}
\item 
the graded skew-symmetry $[x,y] = -(-1)^{|x|\cdot |y|} [y,x]$,
\item
the graded Jacobi identity
$[x,[y,z]]+(-1)^{|x|(|y|+|z|)}[y,[z,x]]+(-1)^{|z| (|x|+|y|)}[z,[x,y]]=0$,
\item
the graded Leibniz rule
$d[x,y]=[d x,y]+(-1)^{|x|}[x,d y]$.
\end{enumerate}

By the construction \eqref{eq:C(V)}, 
for a dg Lie algebra $L$ we have a cocommutative dg coalgebra 
\[
 C'(L) = \bigl(C^{\bullet}(L),d',\Delta\bigr).
\]
Let us rewrite the dg coalgebra structure 
under the isomorphism 
$C^{\bullet}(L) \simeq \oplus_{n \ge 0}\wedge^n(L)[n]$
in \eqref{eq:decalage}.
Let $x_i \in L^{\alpha_i}$ for $i=1,\ldots,n$,
and for $I=\{i_1,\ldots,i_p\} \subset \{1,\ldots,n\}$ with $i_1<\cdots<i_p$ 
we set 
\[
 \oln{I}:= \{1,\ldots,n\} \setminus I,\quad 
 x_I := x_{i_1} \wedge \cdots x_{i_n} \in \wedge^n(S).
\]
We also set $x_{\emptyset} := 1 \in 1_{\M}$.
Then
\begin{equation}\label{eq:chv:coprod}
 \Delta(x_1\wedge \cdots \wedge x_n)=
 \sum_I \sgn(I;\alpha_1,\ldots,\alpha_n) \, x_I \otimes x_{\oln{I}}
\end{equation}
Here $\sgn(I;\alpha_1,\ldots,\alpha_n) \in \{\pm1\}$ 
is determined by the equation
\[
 x_1 \wedge \cdots \wedge x_n = 
 \sgn(I ; \alpha_1,\ldots,\alpha_n) \, x_I \wedge x_{\oln{I}}.
\]
Setting $\wedge^{n}(L) := 0$ for $n<0$, 
the grading structure $C^\bullet(L)$ is described by
\[
 C^n(L) = \oplus_{p+q=n} C^{p,q},\quad
 C^{p,q} := \bigl(\wedge^{-p}(L)\bigr)^q
 = \spn\left\{x_{i_1} \wedge \cdots \wedge x_{i_{-p}} \mid 
         x_{i_k} \in L^{\alpha_{k}}, \ \tsum \alpha_k = q \right\}.
\]
Finally we rewrite the differential $d'$.
The restriction of $d': C^{p+q}(L) \to C^{p+q+1}(L)$ to 
$C^{p,q} \subset C^{p+q}(L)$ is given by 
\begin{equation*}
d': C^{p,q} \to C^{p,q+1},\quad 
x_1 \wedge \cdots \wedge x_{-p} \longmapsto 
(-1)^p \sum_{1 \le i \le -p} 
 x_1 \wedge \cdots \wedge d_{L} x_i \wedge \cdots \wedge x_{-p}.
\end{equation*}

Now we can modify this differential by 
\begin{equation}
\label{eq:Quillen:d}
\begin{split}
&d := d'+d'',\quad
 d'': C^{p,q} \to C^{p+1,q},\\
&d''(x_1 \wedge \cdots \wedge x_{-p}) :=
 \sum_{1 \le i<j \le -p} \sgn(\{i,j\} ; \alpha_1,\ldots,\alpha_n) \, 
 [x_i,x_j] \wedge x_{\oln{\{i,j\}}}.
\end{split}
\end{equation}
Then we have $d^2=(d'')^2=0$.
The complex $C(L)=(C^\bullet(L),d)$ 
is nothing but the Chevalley-Eilenberg complex of the Lie algebra $L$.

\begin{dfn*}
$C(L)=(C^\bullet(L),d,\Delta)$ is called the \emph{Chevalley complex} of $L$.
We always consider it equipped with the filtration 
\begin{equation}\label{eq:Quillen:filtr}
 C(L)_n  = \oplus_{i=0}^n \Sym S^i(L[1])
    \simeq \oplus_{i=0}^n \wedge^i(L)[i].
\end{equation}
The truncated complex
\[
  \oln{C}(L) := (\Sym^{\ge 1}(L),d,\Delta)
\] 
is called the \emph{reduced Chevalley complex} of $L$.
It is equipped with the same filtration 
$\oln{C}(L)_{\bullet}$ as \eqref{eq:Quillen:filtr}
\end{dfn*}

Note that to define the reduced Chevalley complex 
it is not necessary to require the category $\M$ to have a  unit.

\begin{rmk*}
In \cite{HS} $C(L)$ 
is called the Quillen standard complex following \cite[Appendix B]{Q}.
%Precisely speaking the Quillen standard complex is 
%the cocommutative dg coalgebra $\left(\Sym(L[1]),d,\Delta\right)$.
\end{rmk*}

%Following the terminology of , we introduce
%\begin{dfn*}
%The \emph{Quillen standard complex} $C_{\shO_S}(\g)$ of $\g$ 
%is the cocommutative dg coalgebra 
%\[
% C_{\shO_S}(\g) := \left(S_{\shO_S}(\g[1]),d,\Delta\right).
%\]
%We always consider it with the the filtration \eqref{eq:S:filtr},
%namely
%\begin{equation}\label{eq:Quillen:filtr}
% F_n C_{\shO_S}(\g) = \oplus_{i=0}^n S^i_{\shO_S}(\g[1])
% \simeq \oplus_{i=0}^n \wedge^i_{\shO_S}(\g)[i].
%\end{equation}
%\end{dfn*}
%
%\begin{rmk}\label{rmk:Quillen:stdcpx}
%Note that the present construction is valid if we replace $\MO(S)$ by  any 
%$k$-linear symmetric monoidal category $(\calC,\otimes_{\calC})$.
%\end{rmk}

%%%%%%%%%%%%%%%%%%%%%%%%%%%%%%%%%%%%%%%%%%%%%%%%%%%%%%%%%%%%%%%%%%%%%%
\subsection{Homotopy property of Chevalley complex}
\label{subsec:Quillen:hot}

We discuss the homotopy property of the Chevalley complex.
As in the previous subsection, 
let $\M$ be an abelian $k$-linear monoidal category.
Denote by $\CM$ the dg category of complexes in $\M$ as before.

Let us recall the notion of filtered quasi-isomorphisms of complexes.
A \emph{filtered complex} in $\M$ is a complex $C$ in $\M$ 
with an increasing filtration $C_{\bullet}$.
Denote by $\CFM$ the category of filtered complexes $C$ in $\M$ 
such that  $C_n = 0$ for $n \ll 0$ and $C=\cup_i C_i$,
Morphisms in $\CFM$ are those of complexes respecting the filtrations.
A morphism $f:C \to C'$ of filtered complexes 
is called a \emph{filtered quasi-isomorphism} 
if the induced morphism $\gr_i(f):\gr_i(C) \to \gr_i(C')$
on the associated graded is a quasi-isomorphism for any $i$.

Hereafter we use  

\begin{dfn}\label{dfn:Ho}
Denote  by $\Ho\catC$ the homotopy category of 
a closed model category $\catC$.
\end{dfn}

As is well-known, 
the dg category of complexes is a closed model category
with weak equivalences being filtered quasi-isomorphisms and fibrations 
begin surjective morphisms.
The dg category of filtered complexes is also a closed model category
with weak equivalences being filtered quasi-isomorphisms and fibrations 
those morphisms $f$ such that $\gr(f)$ is surjective.

Let us consider the functor $L \mapsto C(L)$ of 
associating the Chevalley complexes to dg Lie algebras $L$ in $\M$.
As explained in the previous subsection, $C(L)$ has a filtration 
so that $C(L) \in \CFM$.
%Now we want to consider the source category.

%Denote by $\opLie(\CM)$ the category of Lie algebras in $\CM$
%(see also Notation in \S\ref{sect:intro}). 
%Namely $\opLie(\CM)$ is the category of dg Lie algebras in $\M$.
%$\opLie(\CM)$ has a closed model category with 
%weak equivalences being quasi-isomorphisms and fibrations 
%those morphisms $f$ such that $\gr(f)$ is surjective.
%In particular, we have the homotopy category $\Ho\opLie(\CM)$.
%
%Let us introduce the following subcategory of $\opLie(\CM)$.

\begin{dfn}\label{dfn:dgla}
Denote by $\opLie(\CM)$ 
the category of dg Lie algebras $L$ in $\M$ such that 
\begin{enumerate}
\item 
every component $L^n$ of the complex $L$ is flat.
Namely, the functor $L^n \otimes_{\M} -$ is exact.
\item
$\calH^n(L)=0$ for $n \gg 0$.
\end{enumerate}
A morphism in $\opLie(\CM)$ 
is that in $\opLie(\CM)$ respecting the Lie brackets.
A morphism $f$ is called a quasi-isomorphism 
if it is a quasi-isomorphism in $\CM$.
\end{dfn}

%Any morphism $f: L \to L'$ in $\opLie(\CM)$ induces a morphism 
%$C(f): C(L) \to C(L')$ in $\CM$.
%By construction $C(L),C(L') \in \CFM$ and 
%$C(f)$ respects the filtration \eqref{eq:Quillen:filtr}.
%Thus $C(f)$ is actually in $\CFM$. 
Now we have

\begin{lem*}%\label{fct:C:fqis}
%If $f$ is a filtered quasi-morphism in $\opLie{\CFM}$
%then the induced morphism $C(f)$ is a filtered quasi-isomorphism.
The functors $\opLie(\CM) \to \CFM$ given by $L \mapsto C(L)$ 
and $L \mapsto \oln{C}(L)$
sends quasi-isomorphism to  filtered quasi-isomorphisms.
\end{lem*}

The proof is as in \cite[\S5.1.4 Lemma]{HS} so we omit it.
Thus the functor $L \mapsto C(L)$ descends to the homotopy category 
and we obtain

\begin{cor*}
There are functors of homotopy categories  
\[
 C, \oln{C}: \, 
 \Ho\opLie(\CM) \longto \Ho\CFM,\quad
 L \longmapsto C(L), \ \oln{C}(L).
\]
\end{cor*}

%%%%%%%%%%%%%%%%%%%%%%%%%%%%%%%%%%%%%%%%%%%%%%%%%%%%%%%%%%%%%%%%%%%%%%
\subsection{Jacobi complex and the universal deformation ring}
\label{subsec:jac:dfn}

In \cite{R1} the Jacobi complex is defined as 
the Chevalley complex of a dg Lie algebra of sheaves on the Ran space 
$\R(X)$ assuming $X$ to be a Hausdorff space.
Here we introduce an analogue for the scheme setting.

Let $X$ be a $k$-scheme.
Recall the convolution product $\otimes^*$ on the category 
$\MO(X^{\catS})$ of $!$-$\shO$-modules in Definition \ref{dfn:conv}.
The resulting tensor category $\MO(X^{\catS})^*$ 
is abelian and $k$-linear.
It has a unit with $\shO := (\shO_{X^I})$ 
with obvious $\theta^{(\pi)}$'s.
%It naturally extends to a tensor structure on 
%the category $\CMO(X^{\catS})$ of complexes of $!$-$\shO$-modules.

%Thus for $V \in \CMO(S^{\catS})$ we have a convolution tensor power 
%\[
% T^n(V) := V \otimes^* \cdots \, \otimes^* V \quad \text{($n$-times product)}.
%\]
%The tensor category $\CMO(S^{\catS})^*$ 
%is symmetric under the same commutator as \eqref{eq:commutator},
%and $\frkS_n$ acts on $T^n(V)$ via this commutator.
%Similarly as in the previous subsection,
%denote by $S^n(V)$ and $\wedge^n(V)$
%the coinvariant and co-anti-invariant spaces respectively.
%We also have a dg coalgebra structure on 
%\[
% C'(V) := \Sym(V[1]) = \oplus_{n\ge0} \Sym^n(V)
%\]
%together with the filtration as \eqref{eq:Quillen:filtr}.

%Under the tensor structure $\otimes^*$, 
%one can consider a Lie algebra object in $\CMO(S^{\catS})^*$.
%%Denote the category of such Lie objects by $\Dgla(S^{\catS})$.
%Let us call such an object 
%\emph{a dg $\shO$-Lie algebra on $S^{\catS}$}.
%As in \S\ref{subsec:Quillen} 
%we can construct Quillen standard complexes for such objects.
%%$\Dgla(S^{\catS})$.

Now we can apply the argument in the previous \S\ref{subsec:jac:chv}  to 
the category $\M = \MO(X^{\catS})^*$.
%Let us simply write 
%\[
% \opBV_u\bigl(X^{\catS}\bigr) := \opBV_u\bigl(\CMO(X^{\catS})^*\bigr),\quad
% \oln{\opBV}\bigl(X^{\catS}\bigr)
%                              := \oln{\opBV}\bigl(\CMO(X^{\catS})^*\bigr),
%\]
A \emph{dg Lie $\shO$-algebra on $X^{\catS}$}
is the complex of $!$-$\shO$-modules in $X^{\catS}$ 
with the Lie algebra structure.

\begin{dfn}
For a dg Lie $\shO$-algebra $L$ on $X^{\catS}$,
we have cocommutative dg coalgebras 
\[
 C(L) = \bigl(\Sym(L[1]),d,\Delta\bigr), \quad
 \oln{C}(L) = \left(\Sym^{\ge1}(L[1]),d,\Delta\right),
\]
where the differential $d$ is given by the formula \eqref{eq:Quillen:d}
and the coproduct $\Delta$ in \eqref{eq:chv:coprod}.
We always consider them with the filtrations 
$C(L)_n := \oplus_{i=0}^n \Sym^i(L[1])$ and 
$\oln{C}(L)_n := \oplus_{i=1}^n \Sym^i(L[1])$.
We call $C(L)$ the \emph{Chevalley complex} and  
$\oln{C}(L)$ the \emph{reduced Chevalley complex} of $L$.
\end{dfn}

We can also discuss the homotopy property as 
in the last paragraph of \S\ref{subsec:BV:chv},
Let us start with 

\begin{dfn*}
Denote by $\opLie(X^{\catS})$ 
the category of dg Lie $\shO$-algebras on $X^{\catS}$
consisting of objects $L=(L_I)$ such that 
\begin{enumerate}
\item 
every component $L_I^n$ of the complex $L_I$ 
is $\shO_{X^I}$-flat for any $I$,
\item
we have the vanishing of the 
the cohomology sheaf $\calH^n(L_I)=0$ for $n \gg 0$ and any $I$.
\end{enumerate}
\end{dfn*}

%A morphism $f: \g \to \frkh$ of dg $\shO$-Lie algebras 
%induces a morphism $C(f):C(\g) \to C(\frkh)$ 
%in the category $\CFMO(S^{\catS})$.
%Now we have a similar claim as Fact \ref{fct:C:fqis}.
%\begin{fct}[{\cite[\S5.1.4 Lemma]{HS}}]
%\label{fct:C:fqis}
%If $f$ is a quasi-morphism in $\Dgla(S)$,
%then the induced morphism $C_{\shO_S}(f)$ 
%is a filtered quasi-isomorphism.

Now we have the next claim
whose proof is the same as in \cite[\S5.1.4 Lemma]{HS}.

\begin{lem*}
If $f$ is a quasi-morphism in $\opLie(X^{\catS})$,
then the image $C(f)$ under the functor $C$ is a filtered quasi-isomorphism.
\end{lem*}

As a corollary we have 

\begin{cor}\label{cor:BV:hot}
The correspondences $L \mapsto C(L)$ and 
$L \mapsto \oln{C}(L)$ yields the functors
\[
 C,\oln{C}: \,
 \Ho\opLie\bigl(X^{\catS}\bigr) \longto 
 \Ho\CMO(X^{\catS}\bigr).
%F_n C:\, 
% \Hola(S^{\catS}) \longto \DFMO(S^{\catS}).
\]
\end{cor}

Now we turn to the definition of the Jacobi complex.
Recall also the fully faithful embedding 
\[
 \Delta^{(\catS)}_*: \MO(X)^* \longinj \MO(X^{\catS})^*
\]
given in Lemma \ref{lem:Delta*S}.
It naturally extends to the embedding 
$\Delta^{(\catS)}_*: \CMO(X) \inj \CMO(X^{\catS})^*$. 
%of the category of complexes.

Let $\g$ be a dg Lie $\shO_X$-algebra,
namely a Lie algebra object in $\CMO(X)^*$.
The image $\Delta^{(\catS)}_*(\g) \in \CMO(X^{\catS})^*$ 
is a Lie object.
In order to ensure $\Delta^{(\catS)}_*(\g) \in \opLie(X^{\catS})^*$,
we consider 

\begin{dfn}\label{dfn:jac:Lie(X)}
Denote by $\opLie(X)$ 
the category of dg Lie $\shO_X$-algebras $\g$ such that 
\begin{enumerate}
\item 
every component $\g^n$ of the complex $\g$ 
is $\shO_{X}$-flat,
\item
$\calH^n(\g)=0$ for $n \gg 0$,
\end{enumerate}
and morphisms are those in $\CMO(X)$ respecting the Lie bracket.
\end{dfn}

If $\g \in \opLie(X)$, then we obviously have
$\Delta^{(\catS)}_*(\g) \in \opLie(X^{\catS})$.
Now we have the main definition.

\begin{dfn}\label{dfn:Jacobi}
For $\g \in \opLie(X)$, 
define the \emph{Jacobi complex} $J(\g)$ to be 
the reduced Chevalley complex of $\Delta^{(\catS)}_*(\g)$.
\[
 J(\g) := \oln{C}(\Delta^{(\catS)}_*(\g)) = 
  \bigl(\Sym^{\ge1}(\Delta^{(\catS)}_*(\g)[1]),d,\Delta\bigr).
\]
Using the filtration on the reduced Chevalley complex, define
\[
 J_n(\g) := \oln{C}\bigl(\Delta^{(\catS)}_*(\g)\bigr)_n.
\]
\end{dfn}

\begin{rmk*}
Originally in \cite{R1} the $n$-th term of the Jacobi complex of $\g$ 
is defined as the $\frkS_n$-anti-invariant part of 
the sheaf $r_{n*}(\g^{\boxtimes n})$ on $\R(X)$. 
Here $r_n:X^n \to \R(X)_n \subset \R(X)$ is the natural projection.
In \cite{R1} $X$ is assumed to be Hausdorff 
so that this definition cannot be compared to ours strictly.
However let us explain that these two definitions 
are essentially the same.

The relation %$r_{m+n}=u_{m,n}\circ(r_m \times r_n)$ in 
\eqref{eq:r=urr}
yields $r_n = v_{1}^n \circ (r_1 \times \cdots \times r_1)$,
where $v_1^n$ is given in \eqref{eq:v}.
Then by the description \eqref{eq:rmk:conv} of $\otimes^*$ 
for Hausdorff $X$ we have 
$r_{n*}(\g^{\boxtimes n}) = 
 (r_{1*} \g)\otimes^* \cdots \otimes^* (r_{1*} \g)$.
In our case  $r_{1*}$ corresponds to  $\Delta^{(S)}_*$,
and taking the anti-invariant part is covered by considering 
$S(\g[1]) \simeq \oplus_n \wedge^n(\g)[n]$.
%Finally the differential in \cite{R1},
%where only a non-dg Lie $\shO_S$-algebra is considered,
%is the same as $d''$ in \eqref{eq:Quillen:d}.
Thus the two definitions are basically the same one.
\end{rmk*}

Note that we have the commutativity of the composition of functors 
$C \circ \Delta^{(\catS)}_* = \Delta^{(\catS)}_* \circ C$.

Let us briefly mention the homotopy property of Jacobi complex 
in a general setting.
A quasi-isomorphism $\g \to \frkh$ in $\opLie(X)$ gives 
a quasi-isomorphism $\Delta^{(S)}_*(\g) \to \Delta^{(S)}_*(\frkh)$.
Then Corollary \ref{cor:BV:hot} gives

\begin{lem}\label{lem:Jacobi:hot}
The functors $\g \mapsto J(\g)$ and $\g \mapsto J_n(\g)$ from $\opLie(X)$
induce those from $\Ho\opLie(X)$.
\end{lem}

Recall the global section functor $\Gamma(X^{\catS},-)$ 
in \eqref{eq:Gamma(X^S)}.
The coproduct $\Delta$ on $J_n(\g)$ induces a ring structure on 
\[
 R^u_n(\g) := k \oplus \left(\Gamma(X^{\catS},J_n(\g))\right)^*, 
\]
making $R^u_n(\g)$ an Artin local $k$-algebra of exponent $n$. 
Since the filtration and $\Delta$ is compatible 
in the sense of Remark \ref{rmk:chv:copro-filtr},
$(R^u_n(\g))_{n\ge1}$ form a direct system of Artin algebras,
and it has the limit
\[
 R^u(\g) := k \oplus \left(\Gamma(X^{\catS},J(\g))\right)^*
         \simeq \varinjlim_n R^u_n(\g).
\]
Note that seen as an object of $\CMO(X^{\catS})$, 
$J_n(\g)$ satisfies the admissible condition in Definition \ref{dfn:adm}.
so we may rewrite
\[
 R^u_n(\g) = k \oplus \left(\Gamma(\R(X),J_n(\g))\right)^*.
\]

Following \cite{R1}, we name 

\begin{dfn*}
For $\g \in \Ho\opLie(X)$, 
we call the Artin algebra $R^u(\g)$
the \emph{universal deformation algebra} of $\g$.
$R^u_n(\g)$ will be called 
the \emph{$n$-th universal deformation algebra} of $\g$.
\end{dfn*}

%%%%%%%%%%%%%%%%%%%%%%%%%%%%%%%%%%%%%%%%%%%%%%%%%%%%%%%%%%%%%%%%%%%%%%%
\subsection{Jacobi complex with coefficients and moduli module}

Recall that for a Lie algebra $\g$ and a $\g$-module $V$
we have the Chevalley complex $C(\g,V)$ with coefficients in $V$.
Thus we can make a similar argument  to obtain 
the Jacobi complex with coefficients in a module.

Let $\M$ be an abelian $k$-linear tensor category as before,
and $L \in \CM$ be a dg Lie algebra in $\M$.
Let also $M$ be a dg $L$-module in $\M$.
Namely $M \in \CM$ with a $\g$-action $\rho:\g \otimes_{\M} V \to V$.
As in the classical case, define the 
\emph{reduced Chevalley complex of $L$ with coefficients in $V$} to be 
\[
 \oln{C}(L,M) := \left(\Hom_{\M}(\oln{C}(L),M),\partial\right).
\]
The differential $\partial$ is given by 
\[
 \partial(g) := \oln{\partial}(g) - (-1)^{|g|}g d
\]
for $g \in \Hom_{\M}(\oln{C}(L),M)$.
Here $d$ is the differential of the complex $\oln{C}(L)$,
and $\oln{\partial}(g)$ is given by the composition
\[ 
 \oln{C}(L) \xrightarrow{\ \Delta \ }
 \oln{C}(L) \otimes_{\M} \oln{C}(L) \xrightarrow{\ \pi \otimes g \ }
 L[1] \otimes_{\M} V \xrightarrow{\ \rho \ } V[1],
\]
where the morphism $\pi: \oln{C}(L) \to L[1]=\Sym^1(L[1])$ 
is the canonical projection.

%The product $m$ on $\oln{C}(L)$ yields an action
%$\oln{C}(L) \otimes \oln{C}(L,V) \to \otimes \oln{C}(L,V)$.

We want to apply this argument to the situation 
in the previous \S\ref{subsec:jac:dfn}.
So let $X$ be a $k$-scheme of finite cohomological dimension,
and $\g \in \opLie(X)$.
Also let $V$ be a $\g$-module in $\CMO(X)^*$,
namely $V$ is a complex of $\shO_X$-modules 
with a $\g$-action $\g \otimes_{\shO_X} V \to V$.
%
%Recall that a $\g$-module structure on $V$ is equivalent to 
%a Lie algebra structure on $\g \oplus V$ such that 
%both the inclusion $\g \inj \g \oplus V$ 
%and the quotient $\g \oplus V \surj \g$ are morphisms of Lie algebras
%and the restriction of the Lie bracket on $V \otimes V$ vanishes.
%
%In order to apply the previous argument to $\g \oplus V$,
%we must assume $\g \oplus V \in \opLie(X)$.
%It is sufficient to consider
%
%\begin{dfn*}
%For $\g \in \opLie(X)$, denote by $\Mod(\g)$
%the full subcategory of $\CMO(X)$ consisting of 
%$\g$-modules $V$ such that 
%\begin{enumerate}
%\item 
%every component $V^n$ of the complex $V$ is $\shO_{X}$-flat,
%\item
%$\calH^n(V)=0$ for $n \gg 0$.
%\end{enumerate}
%\end{dfn*}
%
We set $\M = \MO(X^{\catS})^*$, $L = \Delta^{(\catS)}_*(\g)$
and $M = \Delta^{(\catS)}_*(V)$.
Since $\Delta^{(\catS)}_*$ is a tensor functor by Lemma \ref{lem:Delta*S},
$M$ is an $L$-module.
Thus we have the complex
\[
%J(\g,V) :=  \oln{C}\bigl(\Delta^{(\catS)}_* (\g \oplus V)\bigr),
%\quad
 J_n(\g,V) := 
 \oln{C}\bigl(\Delta^{(\catS)}_*(\g), \Delta^{(\catS)}_*(V)\bigr)_n.
\]
%where $\g \oplus V \in \Lie(X)$ is 
%the Lie algebra respecting the $\g$-module structure as explained above.
It has a coproduct induced by $\Delta$ on the $\g$-factor.
It is natural to name 

\begin{dfn}\label{dfn:jac:wc}
$J_n(\g,V)$ is called 
the $n$-th \emph{Jacobi complex with coefficients in $V$}.
\end{dfn}

The coproduct structure yields 
that the cohomology sheaf $\calH^0(J_n(\g,V))$ on $X^{\catS}$ 
is a sheaf of $R^u_n(\g)$-modules.

Let us define $k$-modules
\[
 M_n(\g,V) := \Gamma\left(X^{\catS},J_n(\g,V)\right).
\]
It is naturally an $R^u_n(\g)$-module.

%Following \cite[\S4.1.6]{BD}, 
%let us spell out the definition in a general setting.
%So let $\M$ be an abelian $k$-linear tensor category 
%as in \S\ref{subsec:BV:dfn},
%and $C$ be a BV algebra in $\M$.
%
%\begin{dfn*}
%A BV $C$-module is a complex $M \in \CM$ together with 
%a BV algebra structure $(m,c)$ on $C \oplus M$ such that 
%\begin{enumerate}
%\item 
%both the inclusion $C \inj C \oplus M$ 
%and the quotient $C \oplus M \surj C$ are morphisms of BV algebras,
%\item
%$m$, hence also $c$, vanishes on $M \subset C \oplus M$.
%\end{enumerate}
%\end{dfn*}
%
%Recall the pair of the functor $\oln{C}:L \mapsto \oln{C}(L)$ and 
%its left adjoint $C \mapsto C_1[-1]$.
%Given $L \in \opLie(\CM)$ and an $L$-module $M$,
%we have the Chevalley complex $C(L\oplus M)$ 
%with $L \oplus M \in \opLie(\CM)$ coming from the $L$-module structure.
%Since $\oln{C}$ is exact, $C(L \oplus M) \simeq C(L) \oplus C(M)$
%we have that inclusion $C(L) \inj C(L \oplus M)$ and the quotient 
%$C(L \oplus M) \surj C(L)$.
%These are morphisms of BV algebras 

\begin{dfn*}
We call the $R^u_n(\g)$-module $M_n(\g,V)$
the $n$-th \emph{moduli module} of $V$.
\end{dfn*}

%%%%%%%%%%%%%%%%%%%%%%%%%%%%%%%%%%%%%%%%%%%%%%%%%%%%%%%%%%%%%%%%%%%%%%
%%%%%%%%%%%%%%%%%%%%%%%%%%%%%%%%%%%%%%%%%%%%%%%%%%%%%%%%%%%%%%%%%%%%%%
\section{Higher Kodaira-Spencer maps}
\label{sect:hks:R}

%We compare the result of \cite{R1} and \cite{HS}.

%%%%%%%%%%%%%%%%%%%%%%%%%%%%%%%%%%%%%%%%%%%%%%%%%%%%%%%%%%%%%%%%%%%%%%
\subsection{The statement}

Let $k$ be  a field  of characteristic $0$, and 
$X$ be a smooth scheme over $k$ 
which is assumed to be separated and quasi-compact.
We study the Jacobi complex of the tangent sheaf $\Theta_X$,
namely 
\[
 J := J(\Theta_X) \supset  J_n := J_n(\Theta_X) 
\]
given by Definition \ref{dfn:Jacobi}.
Let us denote the corresponding universal deformation algebra by 
\[
 R^u_n := R^u_n(\Theta_X)  = k \oplus \left(\Gamma(\R(X),J_n)\right)^*, 
\]
which is an Artin local $k$-algebra of exponent $n$. 

In this section we explain the following result of \cite{R1}.

\begin{thm}\label{thm:hks}
Assume $H^0(X,\Theta_X)=0$.
\begin{enumerate}
\item 
For each $n \in \bbZ_{\ge1}$,
there is a flat deformation $\X^u_n$ of $X$ over $\Spec(R^u_n)$.
The data $\{(\X^u_n,R^u_n)\}_{n \ge1}$  
form a direct system with the limit $(\X^u,R^u)$, 
which is a flat formal deformation $\X^u$ over $\Spf(R^u)$.

\item
$X^u_n$ is universal in the following sense.
For any flat deformation $\X_n$ of $X$ 
over an Artin local $k$-algebra $R_n$ of exponent $n$,
there is a ring homomorphism $\alpha_n: R^u_n \to R_n$
such that $\X_n$ is the pull-back of $X$ by $\alpha_n$.
\end{enumerate}
\end{thm}

\begin{rmk*}
The condition $H^0(\Theta_X)=0$ ensures the existence of moduli space 
of deformations of $X$,
and what we need is in fact the latter condition.
Thus one can consider a weaker condition than $H^0(\Theta_X)=0$
but we omit the details.
\end{rmk*}

The proof will be given in \S\ref{subsec:hks:prf}.
We also have the following result on the differential operators.
Let us denote by $\Diff_Y$ 
the sheaf of $\shO_Y$-differential operators for a scheme $Y$.
$\Diff_Y^{\le n}$ denotes the subsheaf of order $\le n$.

\begin{thm}\label{thm:hks:map}
We have a natural morphism
\[
 \kappa^{\le n}: 
 \Diff^{\le n}_{S^u_n} \longto 
 \shO_{S^u_n} \oplus \calH^0(J_n)
\]
of cocommutative dg coalgebras over $S^u_n := \Spec(R^u_n)$.
Therefore, we also have a morphism
\[
 \kappa: 
 \Diff_{S^u} \longto \shO_{S^u} \oplus \calH^0(J).
\]
\end{thm}

By universality we have

\begin{cor*}[{\cite{HS}}]
For any flat deformation $\X_n$ of $X$ 
over an Artin local $k$-algebra $R_n$ of exponent $n$,
we have a natural morphism
\[
 \Diff^{\le n}_{S_n} \longto 
 \shO_{S_n} \oplus \calH^0(\alpha_n^* J_n)
\]
of commutative dg coalgebras over $S_n := \Spec(R_n)$,
where $\alpha_n: S_n \to \Spec(R^u_n)$ is the morphism in 
Theorem \ref{thm:hks} (2). 
\end{cor*}

The case $n=1$ coincides with the classical Kodaira-Spencer map
\[
 \Diff^{\le 1}_{S_1}/\shO_{S_1} = \Theta_{S_1} 
 \longto \calH^0(\alpha_1^* J_1) \simeq R \pi^1_* \Theta_{\X_1/S_1},
\]
where $\pi:\X_1 \to S_1$ is the canonical projection.
Hence we call $\kappa^{\le n}$ the \emph{higher Kodaira-Spencer map}.

\begin{rmk*}
A similar result as Theorem \ref{thm:hks:map} is shown in \cite[\S3]{R2}.
%We will give an explanation of its  relation to the above theorem 
%in the next \S\ref{sect:G-bundle}..
The construction of higher Kodaira-Spencer maps using Ran space was 
originally announced in \cite{R0}. 
A similar construction without Ran space was given by \cite{EV}.
\end{rmk*}

Before starting the proof of Theorem \ref{thm:hks},
we give two preparations in \S\ref{subsec:hks:MC} and \S\ref{subsec:resol}.

%%%%%%%%%%%%%%%%%%%%%%%%%%%%%%%%%%%%%%%%%%%%%%%%%%%%%%%%%%%%%%%%%%%%%%
\subsection{Maurer-Cartan equation}
\label{subsec:hks:MC}

Let us briefly recall the  Maurer-Cartan equation in dg Lie algebras.
For a dg $k$-Lie algebra $\g=(\g^{\bullet},d,[,])$,
consider a solution $\alpha \in \g^1$ of the Maurer-Cartan equation
\[
 d \alpha + \tfrac{1}{2}[\alpha,\alpha]=0.
\]
We have

\begin{fct}\label{fct:MC}
The twisted differential
\[
 d_\alpha := d + [\alpha,-]
\]
gives a new dg $k$-Lie algebra 
\[
 \g_\alpha :=(\g^{\bullet},d_\alpha,[,]).
\]
\end{fct}

%%%%%%%%%%%%%%%%%%%%%%%%%%%%%%%%%%%%%%%%%%%%%%%%%%%%%%%%%%%%%%%%%%%%%%
\subsection{Resolution}
\label{subsec:resol}

In \cite[\S4]{R1} the universal flat deformation of a complex manifold $X$ 
is constructed over the algebra $R^u=R^u(\Theta_X)$.
The construction is done in terms of the Jacobi complex $J(\shQ(\Theta_X))$ 
with $\shQ(\Theta_X)$ the \v{C}ech or Dolbeault resolution of $\Theta_X$.
%In our setting they correspond to the Dolbeault-style algebras
%$\shQ_{\calU}$ and $\shQ_{\partial}$ in Example \ref{eg:Dolbeault}.

Here we spell out a construction of resolution
via the Thom-Sullivan complex formalism
following \cite[\S4.1.3, Proof of Lemma]{BD} and \cite[\S5.2]{HS}.

For a finite set $I$, denote by $\bbA_I$ the subscheme of 
the affine space $\bbA^I$ defined by the equation $\sum t_i = 1$.
Thus 
\[
 \bbA_I = \Spec(R_I), \quad R_I := k[t_i \mid i \in I]/(\tsum t_i - 1).
\]
Denote by $\Omega_I$ the algebraic de Rham algebra of $R_I$ over $k$.
\[
 \Omega_I := \Gamma(\bbA_I, \Omega_{\bbA_I/k})
 \simeq R_I[ d t_i \mid i \in I  ]/(\tsum d t_i).
\]
It is a commutative dg $k$-algebra with $\deg(dt_i)=1$ and 
$d(t_i)=d t_i$, $d(d t_i) = 0$.
%Denote by $\Omega_I^\bullet$ the dg grading on this algebra.
%
For each $K \subset I$ we have a natural projection 
$\psi: \Omega_I \surj \Omega_K$.
%Also $\Omega=\{\Omega_{[n]}\}_{n\ge0}$ with 
%$[n]:=\{0,1,\ldots,n\}$ is naturally a simplicial commutative dg algebra.

Let $X$ be a separated  quasi-compact $k$-scheme.
Take a finite affine covering $\calU=\{U_s\}_{s \in S}$ of $X$.
For $I \subset S$ denote by 
\[
 j_I: U_I := \cap_{i \in I} U_i \longinj X
\]
the corresponding embedding.
For each $K \subset I$ we have a diagram
\[
 j_{K *} \shO_{U_K} \otimes \Omega_K \xrightarrow{\ \varphi \otimes \id \ } 
 j_{I *} \shO_{U_I} \otimes \Omega_K \xleftarrow{\ \id \otimes \psi    \ }
 j_{I *} \shO_{U_I} \otimes \Omega_I.
\]

Now let us introduce

\begin{dfn*}
Define the subalgebra $\shQ_{\calU}$ of 
the commutative dg $\shO_X$-algebra 
$\prod_{I} j_{I *} \shO_{U_I} \otimes \Omega_I$ to be 
\[
 \shQ_{\calU} := 
 \{ (f_I) \mid (\varphi\otimes \id)(f_K)
    =(\id \otimes \psi)(f_I) \text{ for any } K \subset I\}
 \subset \tprd_{I} j_{I *} \shO_{U_I} \otimes \Omega_I.
\]
\end{dfn*}

$\shQ_{\calU}$ is a \emph{Dolbeault $\shO_X$-algebra} 
in the sense of \cite[\S4.1.3]{BD}.
Namely, it is a commutative unital dg $\shO_X$-algebra, 
quasi-coherent as an $\shO_X$-module, 
satisfying
\begin{enumerate}
\item 
the structure map $\shO_X \to \shQ_{\calU}$ as an $\shO_X$-algebra 
is a quasi-isomorphism,
\item
$\shQ_{\calU}$ is homotopically $\shO_X$-flat, namely,
for every acyclic complex $F$ of $\shO_X$-modules 
the complex $\shQ_{\calU} \otimes_{\shO_X} F$ is acyclic,
\item
$\Spec(\shQ_{\calU}^0)$ is an affine scheme.
\end{enumerate}
By this  remark we have

\begin{lem*}
For a quasi-coherent  $\shO_X$-module $M$,
the canonical map $M \to M \otimes \shQ_{\calU}$ is a quasi-isomorphism.
In particular, the class of $M \otimes \shQ_{\calU}$ 
in the derived category is independent of the choice of $\calU$.
\end{lem*}

\begin{rmk*}
\begin{enumerate}
\item
Let us mention that in \cite[\S5.2]{HS} the same construction is given 
in terms of the injective limit in the category of the diagrams. 
In \cite{HS} $\shQ_{\calU}$ is used for the construction of 
higher Kodaira-Spencer maps.

\item
If  $X$ is smooth and proper over $k=\bbC$,
then the classical Dolbeault algebra $\shQ_{\oln{\partial}}$,
namely the $\oln{\partial}$-resolution of holomorphic functions
\[
 \shQ_{\oln{\partial}} := 
 (\Omega^{0,0}_{X} \xrightarrow{\ \oln{\partial} \ }
 \Omega^{0,1}_{X} \xrightarrow{\ \oln{\partial} \ } \cdots)
\]
is clearly a Dolbeault $\shO_X$-algebra except the quasi-coherent property.
In \cite[\S4.1.4]{BD} such an object is called a Dolbeault-style algebra.
In \cite{R1} $\shQ_{\oln{\partial}}$ 
is used in the construction 
of the universal deformation.

%A \emph{Dolbeault $\shD_X$-algebra} is a dg $\shD_X$-algebra 
%which is a Dolbeault $\shO_X$-algebra.
%It is also a quasi-coherent $\shO_X$-module,
%and indeed is a Dolbeault algebra in the sense of \cite[\S4.1.3]{BD}.
%Namely, for a scheme $X$ over a field $k$,
%a \emph{Dolbeault-style $\shO_X$-algebra} is 
%it is a commutative unital dg $\shO_S$-algebra %$\shQ$
%such that for every quasi-coherent $\shO_S$-module $M$ 
%the natural morphisms 
%\[
% \Gamma(S, M \otimes \shQ) \longto R\Gamma(S, M \otimes \shQ), \quad
% R\Gamma(S, M) \longto R\Gamma(S, M \otimes \shQ)
%\]
%are quasi-isomorphisms.
\end{enumerate}
\end{rmk*}

%%%%%%%%%%%%%%%%%%%%%%%%%%%%%%%%%%%%%%%%%%%%%%%%%%%%%%%%%%%%%%%%%%%%%%
\subsection{The construction of universal family}
\label{subsec:hks:prf}

Let us start the proof of Theorem \ref{thm:hks}.
following \cite{R1}.
Actually  the strategy is based on 
the classical Kodaira-Spencer theory \cite{KS}.

By the assumption on $X$ we can take a finite affine covering 
$\calU=\{U_\alpha\}_{\alpha \in A}$ of $X$.
Set
\begin{equation}\label{eq:gU}
 \g_{\calU} := \shQ_{\calU} \otimes \Theta_X
 = (\g_{\calU}^\bullet,d_{\g},[,]_{\g}),
\end{equation}
where $\shQ_{\calU}$ is given in \S\ref{subsec:resol}.
$\g_{\calU}$ is nothing but the \v{C}ech resolution of $\Theta_X$.
It is a dg Lie algebra satisfying the conditions 
in Definition \ref{dfn:jac:Lie(X)}
By the discussion at Lemma \ref{lem:Jacobi:hot}, 
we have the Jacobi complex $J_n(\g_{\calU})$ and 
\[
 J_n(\g_{\calU}) \longsimto J_n.
\]

Assume that we are given a flat deformation $\X_n$ of $X$ over 
an Artin $k$-algebra $R_n$ of exponent $n$.
Denote  the maximal ideal of $R_n$ by $\frkm_n$.
We have $\frkm_n^{n+1}=0$.

Denoting $\shO := \shO_{X}$ and $\shO_n := \shO_{\X_n}$, 
we have a set $\{\psi^n_{\alpha}\}_{\alpha \in A}$ 
of isomorphisms of algebras
\[
 \psi^n_{\alpha}: \shO_n(U^n_\alpha) \longsimto \shO(U_\alpha) \otimes R_n.
\]
Here $U^n_\alpha$ is the open subset of $\X_n$ corresponding to $U_\alpha$.
Then we can find 
$s_\alpha \in \g_{\calU}^0(U_\alpha) \otimes \frkm_n$ such that 
\[
 \exp(s_\alpha) = \psi^n_\alpha \circ C,
\]
where $C:X \times_{\Spec(k)}\Spec(R_n) \to \X_n$
is a global trivialization
and $\exp(s_\alpha)=\sum_{i=0}^n s_\alpha^i/i!$ is the formal exponential.
Then on $U_{\alpha,\beta} := U_{\alpha} \cap U_{\beta}$ the cocycle
\[
 D^n_{\alpha,\beta} := \psi^n_{\alpha} (\psi^n_{\beta})^{-1}
 \in \Aut_{R_n}\left(\shO(U_{\alpha,\beta}) \otimes R_n\right)
\]
 is expressed by
\[
 D^n_{\alpha,\beta} = \exp(s_\alpha)\exp(-s_\beta).
\]
Below we denote by 
$s := (s_\alpha) \in \g_{\calU}^0 \otimes \frkm_n = \Theta_X \otimes \frkm_n$.

Recall \eqref{eq:Quillen:d} that 
the differential $d$ of the complex $J_n(\g_{\calU})$ 
is given by $d=d'+d''$,
where $d'$ comes from the differential $d_{\g}$ of the \v{C}ech resolution,
and $d''$ comes from the Lie bracket of $\Theta_X$.
Define
\[
 u_n :=  \exp(-s) d'(\exp(s)) = -d'(\exp(-s)) \exp(s).
\]
%Since $C$ is globally defined it commutes with $d'$ and 
%$u_{n,\alpha} = (\psi^n_\alpha)^{-1}d'\exp(\psi^n_{\alpha})$.
The cocycle condition 
$D_{\alpha \beta}^n D_{\beta,\gamma}^n = D_{\alpha \gamma}^n$
yields the Maurer-Cartan equation
\[
 d'(u_n) + \tfrac{1}{2} [u_n,u_n]_{\g} = 0.
\]

Now set 
\[
 v_n := \bigl(u_n,\tfrac{1}{2}(u_n)^2,\ldots,\tfrac{1}{n!}(u_n)^n\bigr)
 \in J_n(\g_{\calU}) \otimes \frkm_n.
\]
One can easily check that the cohomology class
\[
 [v_n] \in 
 \Gamma(X^{\catS},J_n(\g_{\calU})) \otimes \frkm_n 
 \simeq  \Gamma(X^{\catS},J_n) \otimes \frkm_n
\]
depends only on the fixed deformation $\X_n$,
namely is independent of $\calU$, $\{\psi^n_{\alpha}\}$ and $C$.
Thus we have constructed a correspondence
\[
 \X_n/\Spec(R_n) \longmapsto 
 [v_n] \in \Gamma(X^{\catS},J_n) \otimes \frkm_n.
\]

Conversely, starting from a cohomology class 
$[v] \in \Gamma(X^{\catS},J_n) \otimes \frkm_n$,
we can construct a flat deformation as follows.
Take a representative $v_n$ of $[v_n]$ 
and define $u_n \in J^1_n$ to be the degree $1$ of $v_n$.
It satisfies the Maurer-Cartan equation.
Then by Fact \ref{fct:MC} we can modify 
the dg Lie algebra \eqref{eq:gU} to 
\[
 \g'_{\calU} = (\g_{\calU}^\bullet, d_u, [,]_{\g}),\quad
 d_u := d_{\g} + [u,-].
\]
Then
\[
 \shO_n := 
 \ker(d_u:\g_{\calU}^0 \otimes R_n \longto \g_{\calU}^1 \otimes R_n)
\]
is a sheaf of $R_n$-algebras.
It is flat over $R_n$ by the same reason as \cite[Lemma 4.1]{R1}.
Thus we have a flat deformation
\[
 \X_n := \Spec(\shO_n).
\]

The universal family $X^u_n$ is obtained by applying the discussion to
$\id: R^u_n \simto R^u_n$.
The construction respects the filtration 
$J_\bullet$ of $J$, so we have the limit universal family $X^u$
over $\varprojlim \Spec R^u_n = \Spf R^u$.

\begin{rmk*}
In the argument of \cite{R1} an emphasis is put on the OS structure,
which consists of the data on a coalgebra $C$ equivalent to the filtration 
by maximal ideals on the dual Artin algebra $C^*$.
In our formulation, this structure is already built 
in the tensor structure $\CMO(X^{\frkS})^*$ 
where our coalgebra $J_n$  lives.
See also Remark \ref{rmk:conv}.
\end{rmk*}

Before turning to the proof of Theorem \ref{thm:hks:map},
we give a  preparation in the next \S\ref{subsec:connect}.

%%%%%%%%%%%%%%%%%%%%%%%%%%%%%%%%%%%%%%%%%%%%%%%%%%%%%%%%%%%%%%%%%%%%%%
\subsection{Connecting morphism}
\label{subsec:connect}

Let us give a detailed explanation of \emph{unital} coalgebra.
An element $u$ of a $k$-coalgebra $C=(C,\Delta,\ve)$  
is called \emph{group-like} if
$d u = 0$, $\Delta(u)=u \otimes u$ and $\ve(u)=1 \in k$.
A group-like element $u$ defines a splitting 
$C=k u \oplus C^+$ with $C^+ :=\ker(\ve)$.
Denote by $\pi_u:C \surj C^+$ the projection.
Such $u$ also defines a filtration $F^u_\bullet C$
by $F^u_{n-1} C := \ker(C \to C^{\otimes n} \to (C^+)^{\otimes n})$, 
where the first map is the $n$-th composition of $\Delta$,
and the second one is $\pi_u^{\otimes  n}$.
$u$ is called a \emph{unit}
if the filtration $F^u_{\bullet}C$ is exhaustive.
A unital coalgebra is a pair $(C,1_C)$ of coalgebra and its unit.
Similarly one can define a unital coalgebra in any 
unital abelian $k$-linear tensor category $\M$.

In particular, taking $\M=\CMO(X)^*$ with $X$ a $k$-scheme,
we denote by $\Cocom_u(X)$
the category of unital cocommutative dg $\shO_X$-coalgebras.
Let $\g$ be a dg Lie $\shO_X$-algebra and 
$C(\g)$ be its Chevalley complex.
Since the coproduct on $C(\g)$ is cocommutative
and $1 \in C(\g)^0 = \shO_X$ is a unit,
we have $C(\g) \in \Cocom_u(X)$.
For a morphism $f:C \to C(\g)$ of dg $\shO_X$-coalgebras,
we denote by $f_i := p_i\circ f$ 
the composition with the projection 
$p_i:C(\g) \to \Sym^i(\g[1])$.

\begin{fct}[{\cite[Appendix B, 5.3]{Q}}]\label{fct:Q:MC}
We have a bijection  
\[
 \Hom_{\Cocom_u(X)}(C,C(\g)) \longsimto \MC(C,\g), \quad 
 f \longmapsto f_1,
\]
where the target is the space of solutions of Maurer-Cartan equation.
\[ 
 \MC(C,\g) := 
  \left\{f_1 \in \Hom_{\CMO(X)}(C,\g[1]) \mid 
       d f_1 + \tfrac{1}{2}[f_1,f_1]=0, \ f_1(1_C)=0\right\}.
\]
\end{fct}

Next we recall the connecting morphism of Lie algebras
following \cite[\S2.3]{HS}.

Fix a unital abelian $k$-linear tensor category $\M$.
Let $\g$ be a dg Lie algebra in $\M$ and 
$\frkh \subset \g$ be a dg Lie ideal.
Denote by $i:\frkh \inj \g$ the injection and 
set $\C := \Cone(\varphi)$.
Thus $\C$ is a complex with 
\[
 \C^n = \frkh^{n+1} \oplus \g^{n}, 
 \quad
 d_{\C}(x,y) = (-d_{\frkh} x, \varphi(x)+d_{\g}y),
\]
where $d_{\C}$, $d_{\g}$ and $d_{\frkh}$ are 
the differentials of $\C$, $\g$ and $\frkh$.
$\C$ is a dg Lie algebra by
\[
 [(x,y),(x',y')] := 
 \bigl((-1)^p [y,x'] + [x,y'], [y,y']\bigr),\quad
 x,x' \in \frkh,\ y \in \g,\  y \in \g^p.
\]

Define the morphisms $\psi,\pi$ of $\bbZ$-graded objects by 
\[
 \psi: \C \longto \frkh[1],\quad (x,y) \longmapsto x; \qquad
 \pi:  \C \longto \g,\quad (x,y) \longmapsto y.
\]
Note that $\psi$ is a dg morphism but $\pi$ is not.
Denote by $T(V)=\oplus_{n\ge0} V^{\otimes_{\M} n}$ 
the tensor algebra of a complex $V$.
%of $\shO_X$-modules as in \S\ref{subsec:Quillen}.
Define the morphism $\wt{c}:T(\C) \to \g[1]$ of 
$\bbZ$-graded objects inductively by
\[
 \left.\wt{c} \, \right|_{T^0(\C)} :=0,   \quad
 \left.\wt{c} \, \right|_{T^0(\C)} :=\psi,
\]
and for $u \in T^n(\C)$ and $x \in \C$ by 
\[
 \wt{c}(x u) := (-1)^{|x|} [\pi(x), \wt{c}(u)].
\]
Then by \cite[\S2.3.3 Theorem]{HS}
$\wt{c}$ factors through the enveloping algebra $U(\C)$ 
of the Lie algebra $\C$,
and the obtained morphism $U(\C) \to \frkh[1]$ 
satisfies the Maurer-Cartan equation.
Thus by Fact \ref{fct:Q:MC} we have a morphism
\[
 c: U(\C) \longto C(\frkh)
\]
of dg coalgebras to the Chevalley complex of $\frkh$.
Since $U(\C) \simeq U(\g/\frkh)$ as a dg coalgebra,
we have in total 

\begin{fct}[{\cite[\S2.3, \S3.3]{HS}}]\label{fct:HS:conn}
For a dg Lie algebra $\g$ and dg Lie ideal $\frkh \subset \g$,
there is a morphism 
\[
 c: U(\g/\frkh) \longto C(\frkh)
\]
of unital cocommutative dg coalgebras in $\M$.
It is called the \emph{connecting morphism} of the pair $\frkh \subset \g$.
\end{fct}

By the construction, the first order part $c^1$ of $c$ 
is the coboundary map in the long exact sequence
\[
 0 \longto \frkh \longto \g \longto \g/\frkh 
    \xrightarrow{\ c^1 \ } \frkh[1] \longto \cdots,
\]
which is the origin of the name `connecting morphism'.

\begin{rmk*}
Instead of the assumption that $\frkh$ is a dg Lie ideal,
one may ask whether there is a dg Lie algebra structure on 
$\Cone(\varphi)$ of general dg Lie algebra morphism $\varphi$.
\cite{FM} gives the answer that $\Cone(\varphi)$ 
has  no dg Lie algebra structure but has a natural $L_{\infty}$-structure. 
%The dg Lie structure we used is compatible with theirs with 
%$\varphi:\g \surj \g/\frkh$.
Note also that 
\cite{R2,R3} discussed a similar situation called  Lie atoms.
\end{rmk*}

%%%%%%%%%%%%%%%%%%%%%%%%%%%%%%%%%%%%%%%%%%%%%%%%%%%%%%%%%%%%%%%%%%%%%%
\subsection{The construction of higher Kodaira-Spencer maps}
\label{subsec:hks:map}

Let us give a proof of Theorem \ref{thm:hks:map}.
We start with the remark that 
the case $n=1$ recovers the classical Kodaira-Spencer theory 
by the universality property.
Namely, given a first order deformation  $\X_1$ of $X$ over $S_1=\Spec(R_1)$,
we have the Kodaira-Spencer map 
\begin{equation}\label{eq:KS:original}
 \Theta_{S_1} \longto R \pi^1_* \Theta_{\X_1/S_1}
\end{equation}
with $\pi:\X_1 \to S_1$ the projection.
By the classical theory, we know that this map coincides with 
the coboundary map  $c^1$ in 
\[
 0 \longto \pi_* \Theta_{\X_1} \longto \pi_* \Theta_{\X_1} 
   \longto \pi_* \pi^*\Theta_{S_1} \simeq \Theta_{S_1}
   \xrightarrow{\ c^1 \ } R \pi^1_* \Theta_{\X_1/S_1} \longto \cdots 
\]
induced from the short exact sequence
\[
 0 \longto \Theta_{\X_1/S_1} \longto \Theta_{\X_1} 
   \longto \pi^*\Theta_{S_1} \longto 0.
\]
Note also 
$(R\pi^1_* \Theta_{\X_1/S_1})_p \simeq H^1(X,\Theta_X) 
 \simeq \Gamma(X^{\frkS},J^1_n)$,
where $p \in S_1$ corresponds to the original $X$.
($J_n^1$ is the degree $1$ part of the complex $J_n$.)
Thus we have   
\begin{equation}\label{eq:isom:R1=J1}
 R\pi^1_* \Theta_{\X_1/S_1} \longsimto
 \calH ^0(J_n^1 \otimes R_1) = \calH^0(\alpha_1^* J_n^1)
\end{equation}
for any $n \in \bbZ_{\ge1}$. %, 
%where 
%\[
% \alpha_1 := p_n \circ \alpha_n: \, 
% \Spec(R_1) \longto \Spec(R^u_1) \longto \Spec(R^u_n)
%\]
%is the composition of the morphism $\alpha_n$ in Theorem \ref{thm:hks} (2) 
%and the projection $p_1: R^u_n \surj R^u_1$.
Putting \eqref{eq:KS:original} and \eqref{eq:isom:R1=J1} together, 
we obtain
\[
 %\kappa^{\le 1}:=\id \oplus \kappa^1: \, 
 \Diff_{S_1}^{\le 1} = \shO_{S_1} \oplus \Theta_{S_1}
 \longto \shO_{S_1} \oplus \calH^0(\alpha_1^* J_n^1).
\]

%in the previous \S\ref{subsec:hks:prf} 
%coincides with the image $c^1(\partial)$,

Now let us construct a higher analog.
Let $\pi: \X^u_n \to S^u_n=\Spec(R^u_n)$ 
be the universal deformation of order $n$.
Applying  Fact \ref{fct:HS:conn} to 
$\pi_* \Theta_{\X^u_n/S^u_n} \longto \pi_* \Theta_{\X^u_n}$ 
and $\M=\MO(S^u_n)$,
we have a morphism
\[
 c^{\le n}: U(\Theta_{S^u_n}) \longto C(\pi_* \Theta_{\X^u_n/S^u_n}).
\]
On the other hand, we have
\[
 U(\Theta_{S^u_n}) \longsimto \Diff^{\le n}(S^u_n), \quad
 C(\pi_* \Theta_{\X^u_n/S^u_n}) \longsimto
 \shO_{S^u_n} \oplus \calH^0(J_n).
\]
Thus $c^{\le n}$ gives the desired morphism
\[
 \kappa^{\le n}: \Diff^{\le n}(S^u_n) \longto 
 \shO_{S^u_n} \oplus \calH^0(J_n).
\]

%%%%%%%%%%%%%%%%%%%%%%%%%%%%%%%%%%%%%%%%%%%%%%%%%%%%%%%%%%%%%%%%%%%%%%
%%%%%%%%%%%%%%%%%%%%%%%%%%%%%%%%%%%%%%%%%%%%%%%%%%%%%%%%%%%%%%%%%%%%%%
\section{The moduli space of \texorpdfstring{$G$-bundles}{G-bundles}}
\label{sect:hks:G-bundle}

We address an analog of the higher Kodaira-Spencer map in \S\ref{sect:hks:R} 
for the moduli space of $G$-bundles with $G$ an algebraic group.
Our strategy basically follows \cite{HS}.

%%%%%%%%%%%%%%%%%%%%%%%%%%%%%%%%%%%%%%%%%%%%%%%%%%%%%%%%%%%%%%%%%%%%%%
\subsection{Lie algebroid}

Let $X$ be a smooth scheme over $k$.
A (dg) \emph{Lie algebroid over $X$} 
(or (dg) \emph{Lie $\shO_X$-algebroid}) is a sheaf $\shL$ 
of (dg) Lie $k$-algebras on $X$ together with a structure of 
a left $\shO_X$-module and 
a morphism $\tau: \shL \to \Theta_X=\Der_{k}(\shO_X)$ of (dg) Lie $k$-algebras 
and $\shO_X$-modules such that 
$[a,f b ]=f [a,b]+\tau(a)(f)b$ for any $a,b \in \shL$ and $f\in\shO_X$.
The morphism $\tau$ is called the \emph{anchor} of $\shL$.
%Denote $\shL_{(0)} := \ker \tau$,
%which is an $\shO_X$-Lie algebra and a dg Lie ideal in $\shL$.

For a Lie $\shO_X$-algebroid $\shL$, 
a left $\shL$-module is an $\shO_X$-module $M$ with an action of $\shL$ 
as a Lie $k$-algebra with compatibility condition
$l(f m)=l(f)m+f(l m)$ and $(f l)m =f(l m)$ for 
any $f \in \shO_X$, $m \in M$ and $l \in \shL$.
The dg version is similarly defined.

For a Lie algebroid over $X$,
denote by $U_{\shO_X}(\shL)$ the \emph{twisted enveloping algebra}.
Let us recall its definition.
Denote by $U_k(\shL)^{+}$ the augmented ideal 
of the universal enveloping algebra
$U_k(\shL)$ of $\shL$ as a $k$-Lie algebra.
Define $U_{\shO_X}(\shL)^{+}$ to be  the quotient of $U_k(\shL)^{+}$ 
by the two-sided ideal generated by 
$a \cdot f b - f a \cdot b - \tau(a)(f)b$ for 
all $a,b \in \shL$ and $f \in \shO_X$.
Then $U_{\shO_X}(\shL) := \shO_X \oplus U_{\shO_X}(\shL)^{+}$
with the unital algebra structure given by 
$f\cdot a = f a$ and $a \cdot f = f a + \tau(a)(f)$ 
for $a \in \shL$ and $f\in\shO_X$.

$U_{\shO_X}(\shL)$ has a filtration $F_\bullet U_{\shO_X}(\shL)$
coming from the standard one on $U_k(\shL)$.
It also has  a coalgebra structure 
induced by that on $U_k(\shL)$.

\begin{eg}\label{eg:Lalgd:Theta}
We can take $\shL = \Theta_X$.
Then a left $\shL$-module is nothing but a left $\shD_X$-module.
We also have $U_{\shO_X}(\Theta_X) = \Diff_X$,
the sheaf of $\shO_X$-differential operators,  
and $F_n U_{\shO_X}(\Theta_X) = \Diff_X^{\le n}$.
\end{eg}

%%%%%%%%%%%%%%%%%%%%%%%%%%%%%%%%%%%%%%%%%%%%%%%%%%%%%%%%%%%%%%%%%%%%%%
\subsection{Higher Kodaira-Spencer map for Lie algebroid}
\label{subsec:hks:G:seq}

Let us explain the construction of higher Kodaira-Spencer maps 
by \cite[\S7.1]{HS} with the help of Jacobi complexes.
%See also \cite[\S2.9.5]{BD}.

Let $\pi:\X \to S$ be a smooth separated map of schemes over $k$.
We have the short exact sequence
\[
 0 \longto \Theta_{\X/S} \longto \Theta_{\X} \xrightarrow{\ \ve \ } 
 \pi^* \Theta_S \longto 0
\]
Denote by $\pi^{-1}$ the functor of set-theoretical inverse image. 
Hence $\pi^{-1}\Theta_S \subset \pi^*\Theta_S$ 
is a Lie $\pi^{-1}\shO_S$-algebra.
Set  
\begin{equation}\label{eq:Theta_pi}
 \Theta_{\pi} := \ve^{-1}(\pi^{-1} \Theta_S) \subset \Theta_{\X},
\end{equation}
which is the sheaf of vector fields on $\X$ preserving $\pi$.
we have a short exact sequence
\begin{equation}\label{seq:Theta}
 0 \longto \Theta_{\X/S} \longto \Theta_{\pi} \xrightarrow{\ \ve \ } 
 \pi^{-1} \Theta_S \longto 0
\end{equation}
of Lie $k$-algebras and $\pi^{-1}\shO_S$-modules.

Let $\shA$ be a dg Lie algebroid over $\X$ such that 
the anchor $\tau:\shA \to \Theta_{\X}$ is an epimorphism,
namely, the zeroth part $\tau^0:\shA^0 \to \Theta_{\X}$ is surjective.
Setting $\shA_{\X/S} := \tau^{-1}(\Theta_{\X/S})$ and 
$\shA_{\pi} := \tau^{-1}(\Theta_{\pi})$,
we have a short exact sequence
\begin{equation}\label{seq:Lalgd}
 0 \longto \shA_{\X/S} \longto \shA_{\pi} \longto 
 \pi^{-1} \Theta_S \longto 0
\end{equation}
of Lie $k$-algebras and $\pi^{-1}\shO_S$-modules.

Applying the functor $J(-)$ to the exact sequence \eqref{seq:Lalgd},
we have
\begin{equation}\label{seq:Lalgd:J}
 0 \longto J(\shA_{\X/S}) \longto J(\shA_{\pi}) \longto 
 J(\pi^{-1} \Theta_S) \longto 0
\end{equation}
of complexes over $\X^{\catS}$.
On the other hand, we have a canonical adjunction map
\[
 \Theta_S \longto J(\pi^{-1} \Theta_S).
\]
Taking the pull-back of \eqref{seq:Lalgd:J} by this adjunction map,
we have a short exact sequence
\[
 0 \longto J(\shA_{\X/S}) \longto \oln{\shA}_{\pi} \longto 
 \Theta_S \longto 0 
\]
with 
\[
 \oln{\shA}_{\pi} :=
 J(\shA_{\pi}) \otimes_{J(\pi^{-1}\Theta_S)} \Theta_S.
\]
$J(\shA_{\pi})$ and $\Theta_S$ induce on $\oln{\shA}_{\pi}$ 
the structure of dg Lie $k$-algebra and $\shO_{\X}$-module.
Thus we have the pair
\[
 J(\shA_{\X/S}) \longinj \oln{\shA}_{\pi} 
\]
of a dg Lie $\shO_{\X}$-algebra and its dg Lie ideal.

Then we can apply the construction in \S\ref{subsec:connect} 
of the connecting morphism to this pair.
The discussion in \S\ref{subsec:hks:prf} can also be applied to 
the present situation if the universal family exists,
and we have a description of the universal family of deformations 
of $\shA_{\X/S}$.

%%%%%%%%%%%%%%%%%%%%%%%%%%%%%%%%%%%%%%%%%%%%%%%%%%%%%%%%%%%%%%%%%%%%%%
\subsection{The \texorpdfstring{$G$-bundle}{G-bundle} case}
\label{subsec:G-bundle:thm}

Let $X$ be a smooth $k$-scheme, $G$ be a semi-simple algebraic $k$-group,
and $p:P \to X$ be a $G$-torsor over $X$.
Set 
\[
 \shA_{P} := (p^{-1} \Theta_X)^G,
\] 
namely $\shA_{P}$ is the sheaf such that $\shA_{P}(U)$ 
is the space of $G$-invariant vector fields on $p^{-1}(U)$.
$p$ induces a surjection $\tau: \shA_P \to \Theta_X$,
which makes $\shA_{P}$ a Lie algebroid over $X$.
If $H^0(X,\shA_P)=0$ and 
there is a universal deformation of $(X,P)$.

Now the argument in the previous subsection can be applied to 
the Lie algebroid $\shA_{P}$.

\begin{thm}[{\cite{HS}}]\label{thm:G-bundle}
Assume $H^0(X,\shA_P)=0$.
Set
\[
 R^u_n := k \oplus \Gamma(\R(X),J_n(\shA_{P})).
\]
\begin{enumerate}
\item 
For each $n \in \bbZ_{\ge1}$,
there is a flat deformation $\Y^u_n$ of $(X,P)$ over $\Spec(R^u_n)$.
The data $\{(\Y^u_n,R^u_n)\}_{n \ge1}$  
form a direct system with the limit $(\Y^u,R^u)$, 
a flat formal deformation $\shP^u$ over $\Spf(R^u)$.
Moreover $\Y^u_n$ is universal 
in the sense of Theorem \ref{thm:hks} (2).

\item
%For any flat deformation $\shP_n$ of $(X,P)$ 
%over an Artin local $k$-algebra $R_n$ of exponent $n$,
We have a natural morphism
\[
 \kappa^{\le n}: 
 \Diff^{\le n}_{S^u_n} \longto 
 \shO_{S^u_n} \oplus \calH^0(J_n(\shA_{P}))
\]
of cocommutative dg coalgebras over $S^u_n := \Spec(R^u_n)$.
\end{enumerate}
\end{thm}

We have an $X$-fixed version, namely 
the deformation problem of the $G$-torsors over a fixed scheme $X$.
Then we can take the Jacobi complex $J(\g_P)$ of 
\[
 \g_{P} := \ker(\tau: \shA_P \longto \Theta_X),
\]
and apply the same argument.
%As a result we have the same statement as in the above theorem.

%%%%%%%%%%%%%%%%%%%%%%%%%%%%%%%%%%%%%%%%%%%%%%%%%%%%%%%%%%%%%%%%%%%%%%
%%%%%%%%%%%%%%%%%%%%%%%%%%%%%%%%%%%%%%%%%%%%%%%%%%%%%%%%%%%%%%%%%%%%%%
\section{Flat connection on the Jacobi homology}
\label{sect:BV}

In \cite{R2} Ran gave a general construction of (projective) flat connections
related to the universal deformation.
His strategy can be stated in the following steps. 
\begin{enumerate}
\item 
Translate the heat equation in Hitchin's construction \cite{H} of 
projective flat connections
into the language of dg Lie algebras on the Ran space.
As a result one obtains the ``connection algebra",
which is the Jacobi complex of a certain Lie atom. 
 
\item
The connection algebra has a canonical trivialization over 
the universal deformation algebra.
This trivialization gives the desired connections.
\end{enumerate}

In this note we give another approach.
Our strategy is to take an analog of flat connections 
on chiral homology due to Beilinson and Drinfeld in their theory of 
chiral algebra \cite[\S4]{BD}.
Let us briefly explain their argument.

Chiral algebras are certain Lie algebra objects 
in a non-standard tensor category of $\shD$-modules 
on a fixed curve $X$.
One can consider a kind of reduced Chevalley complex 
$\oln{C}(A)$ (called the Chevalley-Cousin complex) of a chiral algebra $A$.
$C(A)$ can be seen as a $\shD$-module on the Ran space $\R(X)$.
Chiral homology $C^{\ch}(A)$ is defined to be $R\Gamma(\R(X),\DR(C(A)))$,
where $\DR(-)$ is the de Rham functor on $\R(X)$ 
(we omit the precise definition).
$C^{\ch}(A)$ is  a generalization of the conformal block 
in conformal field theory. 

The main ingredient of their argument 
is the BV algebra structure on the Chevalley complex $\oln{C}(A)$,
which yields a canonical trivialization on Lie algebra action on $\oln{C}(A)$.
A certain set of assumptions called ``the package" in \cite[\S4.5]{BD}
enables one to construct flat connections on $C^{\ch}(A)$ using 
this canonical trivialization. 

One finds that their situation looks similar to ours,
namely they handle Lie algebra objects on the Ran space.
Moreover, the BV algebra structure exists for the Chevalley complex 
of Lie algebra in any tensor category, 
as we explain in \S\ref{subsec:BV:dfn}.
Thus it is natural to expect that one can take an analog of 
the arguments in \cite{BD}.

%%%%%%%%%%%%%%%%%%%%%%%%%%%%%%%%%%%%%%%%%%%%%%%%%%%%%%%%%%%%%%%%%%%%%%
\subsection{BV algebras}
\label{subsec:BV:dfn}

Following \cite[4.1.6]{BD}, we recall the notion of 
\emph{Batalin-Vilkovisky algebras}, or \emph{BV algebras} for short.
See also \cite[\S13.7]{LV}, 
although its presentation has some minor difference from ours.

Roughly speaking, a BV $k$-algebra $C$  is a 
$1$-Poisson (or Gerstenhaber) dg $k$-algebra.
%namely, a dg $k$-algebra structure on $C$ together with 
%a Poisson algebra structure on $C[-1]$.
Since we will later consider BV algebras 
in several different tensor categories,
let us spell out the definition using the language of operads.

\begin{dfn*}
The \emph{BV operad $\opBV$} is defined to be a dg $k$-operad 
which is inhomogeneous quadratic (in the sense of \cite[\S7.8]{LV})
generated by the differential $d$ of degree $1$, 
the product $m= (-\cdot-)$ of degree $0$
and the $1$-Poisson bracket $c=\{-,-\}$ of degree $1$.
The relations consists of $d^2=0$, 
the $1$-Poisson operad relation for $m$ and $c$,
and the $c = d \circ m$.
\end{dfn*}

Let $\M$ be an abelian $k$-linear tensor category 
with the tensor structure $\otimes_{\M}$.
%
%Let $\M=(\M,\otimes_{\M})$ be as mentioned in the beginning of this section. 
We denote by $\CM$ the dg category of complexes in $\M$.
The induced tensor product on $\CM$ is 
denoted by the same symbol $\otimes_{\M}$.
Using the framework of algebras over operad 
(see \cite[Chap.\ 5]{LV} for example), we have 

\begin{dfn*}
A \emph{BV algebra in $\M$} means a complex $C$ in $\M$
together with a morphism $\opBV \to \opEnd_C$ of dg operads,
where $\opEnd_C := \oplus_n \Hom_{\CM}(C^{\otimes_{\M} n},C)$ 
is the endomorphism operad on $C$.
\end{dfn*}

For example, letting $\M$ be the category of $k$-vector spaces,
a BV $k$-algebra $C$ is a complex $(C,d_C)$ of $k$-vector spaces 
together with a Poisson structure on the graded vector space $C[-1]$
consisting of the product $\cdot_C$ and the Poisson bracket $\{,\}_C$.
These data should satisfy the relation 
$\{,\}_C = d_C \cdot_C - \cdot_C d_{C \otimes C}$.
This is what we mentioned roughly in the beginning.

Let $\M$ be an abelian $k$-linear tensor category again,
and $C=(C,d,m,c)$ be a BV algebra in $\M$.
Then $L := C[-1]$ is naturally a Lie algebra in 
the category $\CM$ %of complexes in $\M$ 
with the Lie bracket $c$,
and the BV structure yields an $L$-action on $C$.
Then we set 
\begin{equation}\label{eq:bv:dagger}
 L_{\dagger} := \Cone(\id_L) \in \CM,
\end{equation}
which is a contractible complex.
As a $\bbZ$-graded object we have 
$L_{\dagger} \simeq L[1] \oplus L = C \oplus L$.
$L_{\dagger}$ is naturally a Lie algebra in $\CM$ 
(see \S\ref{subsec:connect}),
and the $L$-action on $C$ extends to the $L_{\dagger}$-action 
with the component $L[1]=C$ acting on $C$ by $m$.

%%%%%%%%%%%%%%%%%%%%%%%%%%%%%%%%%%%%%%%%%%%%%%%%%%%%%%%%%%%%%%%%%%%%%%
\subsection{Filtered BV algebras}
\label{subsec:BV:filt}

Let $\M$ be an abelian $k$-linear tensor category as before.
One can consider $\opBV$ as a dg filtered operad
with the increasing stupid filtration $\opBV_n := \opBV^{\ge -n}$. 
Then one has 

\begin{dfn*}
A \emph{filtered BV algebra} in $\M$ is a BV algebra $C$ in $\M$ 
together with an increasing filtration $C_{\bullet}$ 
which is compatible with the BV algebra structure.
Denote by $\opBV(\M)$ the category of 
filtered BV algebras $C$ in $\M$
such that $C_{-1}=0$ and $\cup_{n \in \bbZ} C_n = C$.
Also let $\oln{\opBV}(\M)$ be the full subcategory in $\opBV(\M)$ 
consisting of objects $C$ such that $C_0=0$.
\end{dfn*}

%In particular, 
%the associated graded $\gr_{\bullet}(C)$ is a $1$-Poisson algebra.

One can naturally augment $\opBV$ and obtains a dg operad $\opBV_u$
which encodes the structure of unital BV algebras.
In other words, we introduce

\begin{dfn*}
Assume that $\M$ has a unit and is a symmetric monoidal category.
We define a \emph{unital BV algebra} $C$ in $\M$ to 
be a BV algebra $(C,d,m,c)$ in $\M$ 
having a unit $1 \in C^0$ with respect to $m$ such that $d(1)=0$.
In the filtered setting we assume $1 \in C_0$.
The subcategory in $\opBV(\M)$  consisting of unital filtered BV algebras 
is denoted by $\opBV_u(\M)$.
\end{dfn*}

By \cite[\S4.1.7.\ Proposition]{BD},
$\opBV(\M)$, $\oln{\opBV}(\M)$ and $\opBV_u(\M)$ are closed model categories 
with weak equivalences being filtered quasi-isomorphisms 
and fibrations being morphisms $f$ such that $\gr(f)$ is surjective.
Thus we have the corresponding homotopy categories
$\Ho\opBV(\M)$, $\Ho\oln{\opBV}(\M)$ and $\Ho\opBV_u(\M)$
(recall Definition \ref{dfn:Ho}).

%%%%%%%%%%%%%%%%%%%%%%%%%%%%%%%%%%%%%%%%%%%%%%%%%%%%%%%%%%%%%%%%%%%%%%
\subsection{Chevalley complex}
\label{subsec:BV:chv}

Let $\M$ be as before.
For a filtered BV algebra $C$ in $\M$, 
the shifted filter $C_1[-1]$ is naturally a dg Lie algebra in $\M$ 
with respect to $c=\{,\}$.
The correspondence $C \mapsto C_1[-1]$ yields functors
\[
 \oln{\opBV}(\M) \longto \opLie(\CM), \quad 
 \opBV_u(\M)     \longto \opLie(\CM),
\]
where $\opLie(\CM)$ denotes the category of dg Lie algebras in $\M$.
In the case of $\opBV_u(\M)$, one should suppose $\M$ to have a unit.

\begin{rmk*}
If we have left adjoints of these functors,
then they will give us examples of BV algebras automatically.
Indeed by \cite[\S4.1.8]{BD} we have left adjoints, 
and they are nothing but constructing the Chevalley complexes
as the title of this subsection implies.
Let us give an explanation of the way to find this answer.

Recall that the forgetting functor $V_0$ from 
Poisson algebras (say over $k$) to
Lie algebras has a left adjoint $S$ assigning to a Lie algebra $L$ 
the symmetric algebra $\Sym(L)$ together with the Kostant-Kirillov 
Poisson bracket.
The adjoin pair $(S,V_0)$ is the ``classical part'' of 
the following pair $(U,V)$.
$V$ is the forgetting functor from associative algebras to 
Lie algebras with the same vector space and 
the commutator as the Lie bracket.
$U$ is the functor assigning to a Lie algebra $L$ 
the universal enveloping algebra $U(L)$.
Since  associative algebra degenerate to Poisson algebras,
and since $U(L)$ is a deformation quantization of $S(L)$,
we can say $(S,V_0)$ is the classical part of $(U,V)$. 

Thus one can guess that the desired left adjoints 
are given by $L \mapsto \Sym(L[1])$, where 
the shift $[1]$ is necessary because we are considering 
a $1$-Poisson structure.
$\Sym(L[1])$ is nothing but the Chevalley complex of $L$.
\end{rmk*}

By \cite[\S4.1.8]{BD}, the functors $C \mapsto C_1[-1]$ 
have left adjoints
\[
 \oln{C}:\, \opLie(\CM) \longto \oln{\opBV}(\M),\quad
 C:      \, \opLie(\CM) \longto \opBV_u(\M),
\]
where for $L \in \opLie(\CM)$ the corresponding 
$C(L)$ is given by the Chevalley complex of $L$,
and $\oln{C}(L)$ is the reduced Chevalley complex of $L$.
Thus as $\bbZ$-graded objects in $\M$ we have 
\[
 C(L) = \Sym(L[1]) = \oplus_{n\ge0}\Sym^n(L[1]),\quad
 \oln{C}(L) = \Sym^{\ge 1}(L[1]).
\]
The filtration is given by $C(L)_n := \Sym^n(L[1])$.
The differential and the $1$-Poisson bracket are determined by 
the condition that the embedding 
$L = \Sym^1(L[1])[-1] \inj C[-1]$ 
is a morphism of  dg Lie algebras.
This BV structure respects the filtration.

These functors preserve filtered quasi-isomorphisms,
so that they descent to homotopy categories
and yield adjoint pairs 
\[
 \Ho\opLie(\M) \xrightleftharpoons[]{\ \ \oln{C} \ \ } \Ho\oln{\opBV}(\M),
 \quad
 \Ho\opLie(\M) \xrightleftharpoons[]{\ \ C \ \ } \Ho\opBV_u(\M)
\]
on homotopy categories.

%\[
%a \xrightleftharpoons[unten]{oben} b \rightleftharpoons c
%\]

As a corollary, we find that the Jacobi complex has a BV structure.

%%%%%%%%%%%%%%%%%%%%%%%%%%%%%%%%%%%%%%%%%%%%%%%%%%%%%%%%%%%%%%%%%%%%%%
\subsection{Rigidity}
\label{subsec:BV:rigid}

Let $X$ be a separated quasi-projective $k$-scheme and $\g,L \in \opLie(X)$ as 
in Definition \ref{dfn:jac:Lie(X)}.

Suppose $L$ acts on $\g$.
Then $\Gamma(X,L)$ acts on $\g$ by derivation,
and further $\Gamma(X,L)$ acts on the Jacobi complex $J(\g)$.
One can replace $\g$ and $L$ by their resolutions.
For example, take the \v{C}ech complex $\shQ$ in \S\ref{subsec:resol}
and consider $\g_{\shQ} := \g \otimes_{\shO_X} \shQ$, and 
$L_{\shQ} := L \otimes_{\shO_X} \shQ$ instead of $\g$ and $L$.
We still have a $\Gamma(X,L_{\shQ})$-action on 
$J(\g_{\shQ}) = J(\g)_{\shQ}$.
Since $L_{\shQ}$ is canonically identified with $L$ in $\Ho\opLie(X)$,
$\Gamma(X,L_{\shQ})$ is identified with $R \Gamma(X,L)$.
Thus the homotopy Lie algebra $R \Gamma(X,L)$ acts on $J(\g)$.

Now we have an analog of the rigidity property of Chevalley-Cousin complex 
in \cite[\S 4,5.2]{BD}.

\begin{lem}\label{lem:bv:rigid}
Suppose that the $L$-action comes from 
a Lie algebra morphism $\iota: L \to \g$.
Then the $R\Gamma(X,L)$-action on $J(\g)$ is canonically 
homotopically trivialized.
\end{lem}

\begin{proof}
Recall the BV structure on $J(\g)$.
The morphism $\iota$ yields a morphism 
\[
 \Gamma(X,L_{\shQ}) \xrightarrow{\ \iota \ }
 \Gamma(X,\g_{\shQ}) \longinj (J(\g)_{\shQ})[-1]
\]
of Lie algebras.
So the Lie algebra $\Gamma(X,L_{\shQ})_{\dagger}$
(see \eqref{eq:bv:dagger} and around)
acts on $J(\g)$ via the canonical action of 
$\left(J(\g)_{\shQ}[-1]\right)_{\dagger}$ given by the BV structure.
Since $\Gamma(X,L_{\shQ})_{\dagger}$ is contractible,
we obtain a homotopy from the given action to a trivial action.
\end{proof}

%Following \cite{R2} and \cite[\S3]{R3}, 
%we introduce
%
%\begin{dfn}
%$\g \in \opLie(X)$ is said to \emph{have central extensions} 
%if for each open $U \subset X$ the image of the restriction map 
%$\Gamma(X,\g) \to \Gamma(U,\g)$ is contained 
%in the centre of $\Gamma(U,\g)$.
%\end{dfn}

Let us give a variant of this statement.
Suppose that $\g \in \opLie(X)$ is  a central extension 
\[
 0 \longto \shO_X \longto \g \longto \frkh \longto 0
\]
of dg Lie $\shO_X$-algebras and that 
$L \in \opLie(X)$ acts on $\g$ by the Lie homomorphism 
$\oln{\iota}: L \to \frkh$ and the adjoint action of $\frkh$ on $\g$.
Then we have $R\Gamma(X,L)$-action on $J(\g)$ similarly as above.

\begin{lem}\label{lem:bv:rigid:ext}
In this situation, the $R\Gamma(X,L)$-action is homotopically 
equivalent to the multiplication of a character.
\end{lem}

\begin{proof}
Denote by $L^{\flat}$ the $\shO_X$-extension  of $L$ 
defined as the pull-back of $\g \surj \frkh$ by $\iota$.
The resulting Lie algebra morphism $\ve: L^{\flat} \surj L$ 
yields a dg Lie algebra $L^{\flat}_{\dagger} := \Cone(\ve)$.
As in the previous Lemma \ref{lem:bv:rigid},  
the Lie algebra 
$\wt{L} := R\Gamma(X,L_{\dagger}^{\flat})$ acts on $J(\g)$,
and it is homotopic to $R\Gamma(X,L)$.
In the homotopy category $\wt{L}$ is equivalent to $k$.
Thus we are done.
\end{proof}

\subsection{Relative rigidity}
\label{subsec:BV:rel-rigid}

We want to discuss a relative version of the rigidity property 
explained in \S \ref{subsec:BV:rigid}.
Our presentation is an analog of 
``the package" of Lie algebroid action on a chiral algebra governed by 
a $\Lie^*$ algebra action 
in \cite[\S\S4.5.4, 4.5.5]{BD}.

Let $\pi:\X \to S$ be a smooth proper flat family of $k$-schemes.
For a point $s \in S$ we denote by $\X_s$ the corresponding fiber.
All the notions explained in the previous sections have a natural relative version.
For example, the Ran space $\R(\X/S)$ is a space fibered over $S$  
with the fiber the usual Ran space $\R(\X_s)$.
The symbol $\opLie(\X/S)$ denotes the relative version of 
Definition \ref{dfn:jac:Lie(X)}.
Namely the category of $\shO_S$-flat $\shO_S$-modules $\g$ such that 
$\g_s$ is a dg Lie $\shO_{\X_s}$-algebras for every $s \in S$
satisfying the two conditions similar as in Definition \ref{dfn:jac:Lie(X)}.
The objects in $\opLie(\X/S)$ will be called 
the \emph{flat family of dg Lie algebras on $\X/S$}.
We also have the notion of \emph{flat family of dg Lie algebroids on $\X/S$}.
The \emph{relative Jacobi complex} $J(\g)$ for $\g \in \opLie(\X/S)$ 
is naturally defined.
Finally 
let us denote by $R\pi_*(\R(\X/S),-)$ the relative version of the functor 
$R\Gamma(\R(X),-)$ defined in \S\ref{subsec:Ran:XS}.

Let $\shL$ be a Lie algebroid over $S$ with 
$\tau:\shL \to \Theta_S$ the anchor.
Define 
\[
 \pi^{\sharp}\shL := \pi^{-1}\shL \otimes_{\pi^{-1}\Theta_S} \Theta_{\pi}.
\]
Here $\Theta_{\pi}$ is the subsheaf in $\Theta_{\X}$ consisting of 
vector fields preserving $\pi^{-1}\shO_{S} \subset \shO_{\X}$
(see also \eqref{eq:Theta_pi} and \eqref{seq:Theta}).
Hence $\pi^{\sharp}\shL$ is a Lie $\pi^{-1}\shO_{S}$-algebroid 
acting on $\shO_{\X}$, and sits in the exact commutative diagram 
\[
 \xymatrix{
    0 \ar[r]                
  & \Theta_{\X/S}    \ar[r] \ar@{=}[d] 
  & \pi^{\sharp}\shL \ar[r] \ar[d]
  & \pi^{-1}\shL     \ar[r] \ar[d]^{\pi^{-1} \tau}  &  0
 \\
    0 \ar[r]  
  & \Theta_{\X/S}      \ar[r]  
  & \Theta_{\pi}       \ar[r] 
  & \pi^{-1}\Theta_{S} \ar[r] 
  &  0
 }.
\]
Let us also define $\pi^{\dagger}\shL$ to be the push-out 
of $\shO_{\X} \otimes_{\pi^{-1}\shO_S} \pi^{\sharp}\shL$
by the product map 
$\shO_{\X} \otimes_{\pi^{-1}\shO_{S}} \Theta_{\X/S} \to \Theta_{\X/S}$.
We have a short exact sequence
\[
 0 \longto \Theta_{\X/S} \longto \pi^{\dagger} \shL 
   \longto \pi^*\shL \longto 0
\]
of $\shO_{\X}$-modules,
If $M$ is a left $\shL$-module, then $\pi^* M$ is naturally 
a left $\pi^{\dagger}\shL$-module.

Mimicking the setting in \cite[\S4.5.4]{BD}, 
let us suppose that the following data is given.
\begin{enumerate}[\quad \ \ (a)]
\item 
a dg Lie algebroid $\shL$ on $S$,
\item
a Lie $\pi^{-1}\shO_{S}$-algebra $L$,
\item
a dg Lie $\pi^{-1}\shO_S$-algebroid extension $\shK$ 
of $\pi^{\sharp}\shL$ by $L$, 
\item
a section $s:\Theta_{\X/S} \to \shK$,
\item
a left $\shK$-module structure on $L$
(where $\shK$ is seen as a Lie $\pi^{-1}\shO_S$-algebroid),
\item
a Lie algebroid $\shA$ on $\X/S$ which is $\shO_S$-flat.
\item
a left $\shK$-module structure on $\shA$.
\item
a  morphism $\iota: L \to \shA$ of dg Lie algebras on $\X/S$.
\end{enumerate}
These should satisfy
\begin{enumerate}
\item 
The $\shK$-actions on $L$ and $\shA$ are compatible with 
the Lie brackets on $L$ and $\shA$.

\item
$\iota:L \to \shA$ commutes with the $\shK$-actions.

\item
$s(\Theta_{\X/S}) \subset \shK$ is a Lie ideal.
Hence $\shK$ is an extension of $\pi^{-1}\shL$ 
by $\shK_0 := L \otimes_k \Theta_{\X/S}$.

\item
$\Theta_{\X/S} \subset \shK_0$ acts on $\shA$ and $L$ trivially.

\item
$L \subset \shK_0$ acts on $L$ via the adjoint action, 
and on $\shA$ via $\iota$ and the adjoint action.

\item
The adjoint action of $\shK$ on $L \subset \shK_0$ 
coincides with the $\shK$-action on $L$ coming from 
the left $\shK$-module structure.
\end{enumerate}

\begin{thm}
There is a homotopy left $\shL$-module structure on 
$R^i \pi_*(\R(\X/S),J(\shA))$ for each $i$.
In particular, if $\shL$ is a plain (non dg) Lie $\shO_S$-algebroid,
then $R^i \pi_*(\R(\X/S),J(\shA))$ 
are left $\shL$-modules.
\end{thm}

\begin{proof}
By the condition (3) we have a short exact sequence
\begin{equation}\label{seq:K}
 0 \longto \shK_0 = L \otimes \Theta_{\X/S} \longto \shK 
   \longto \pi^{-1}\shL \longto 0
\end{equation}
The strategy is to obtain a homotopy $\shL$-action from
the $R\pi_{*}\shK$-action  (given by the sixth data)
by trivializing homotopically the action of $R\pi_{*}\shK_0$ 
as in Lemma \ref{lem:bv:rigid}.

Let $\shQ$ be a Dolbeault $\shO_S$-algebra equipped 
with a left $\pi^{\dagger}\shL$-action
(which always exists by \cite[\S4.5.5 Proof (i) (b)]{BD}).
We can replace $L$, $\shL$, $\shA$ and so on by 
homotopy Lie algebras (algebroids)  $L_{\shQ} := L \otimes \shQ$,
$\shL_{\shQ}$, $\shA_{\shQ}$ and so on 
preserving the conditions.

By the condition (4), the $\shK_0$-action on $A$ factor through 
an action of 
$\wt{\shK}_{\shQ} := (\shK_{0} \otimes \shQ)/(\Theta_{\X/s} \otimes \shQ)$.
Then the sequence \eqref{seq:K} yields
\[
 0 \longto \pi_*  L_{\shQ}  \longto \pi_* \wt{\shK}_{\shQ} 
   \longto \shL_{\shQ} \longto 0 
\]
The $\wt{\shK}$-action induces a $\pi_*\wt{\shK}$-action on 
$R\pi_*(\R(\X/S),J(\shA))$.

Set $L_{\dagger} := \Cone(\id:L \simto L)$,
which is a contractible Lie algebra.
There is also a natural embedding $L_{\shQ} \inj L_{\dagger} \otimes \shQ$.
Now define $\wt{\shL}_{\shQ}$ to be the push-out of $\pi_* \wt{\shK}_{\shQ}$ 
by this embedding.
$\wt{\shL}$ sits in the exact commutative diagram.
\[
 \xymatrix{
    0 \ar[r] 
  & \pi_* L_{\shQ}        \ar[r] \ar[d] 
  & \pi_*\wt{\shK}_{\shQ} \ar[r] \ar[d]
  & \shL_{\shQ}           \ar[r] \ar[d]^{\tau}  &  0
 \\
    0 \ar[r]  
  & \pi_* L_{\dagger} \otimes \shQ \ar[r]  
  & \wt{\shL}_{\shQ}               \ar[r] 
  & \Theta_S \otimes \shQ          \ar[r] 
  & 0
 }.
\]
Since $\pi_* L_{\dagger} \otimes \shQ$ is contractible, 
we find that a left $\pi_* \wt{\shL}_{\shQ}$-module is equivalent to 
a left $\pi_* \shL_{\shQ}$-module.

The $\pi_*\wt{\shK}_{\shQ}$-action on $R\pi_*(\R(\X/S),J(\shA))$ 
induces a left $\pi_* \shL_{\shQ}$-module structure 
on $R\pi_*(\R(\X/S),J(\shA))$,
which is the desired homotopy left action of $\pi_* \shL$.
\end{proof}

We can consider a variant of this theorem as in the previous subsection.
Suppose we are given the data (a)--(e), (g) as before and 
\begin{enumerate}[\quad \ \ (a')]
\setcounter{enumi}{5}
\item 
an $\shO_S$-flat dg Lie algebroid $\shA$ on $\X/S$ which is 
a central extension Lie $\shO_{\X/S}$-algebra, 
\setcounter{enumi}{7}
\item
a morphism $\oln{\iota}: L \to \shA/\shO_{\X/S}$ of Lie algebras.
\end{enumerate}
These should satisfy the conditions (1)--(6) 
with $\iota$ replaced by $\oln{\iota}$.

Then we have an $\shO_{\X/S}$-extension $L^{\flat}$ of $L$ 
and a morphism of Lie algebras 
$\iota^{\flat}: L^{\flat} \to \shA$ lifting $\oln{\iota}$.
$\iota^{\flat}$ commutes with the $\shK$-action on $L^\flat$ 
which is the lift of the action (e).

\begin{thm}
There is a homotopy $\shO_S$-extension $\shL^{\flat}$ of $\shL$ 
and a left $\shL^{\flat}$-module structure on 
$R \pi_*(\R(\X/S),J(\shA))$.
\end{thm}

\begin{proof}
The proof is similar with $L_{\dagger}$ replaced by 
$L^\flat_{\dagger}:=\Cone(L^{\flat} \surj L)$
See also the proof of Lemma \ref{lem:bv:rigid:ext}.
\end{proof}

%%%%%%%%%%%%%%%%%%%%%%%%%%%%%%%%%%%%%%%%%%%%%%%%%%%%%%%%%%%%%%%%%%%%%%
\subsection{Hitchin connection}

Let us apply the discussion in the last \S \ref{subsec:BV:rel-rigid}
to the  relative version of  \S\ref{subsec:G-bundle:thm}.

Let $G$ be a semi-simple algebraic $k$-group and 
$\calC \to S$ be a flat family of smooth $k$-curves of 
genus greater than $2$.
Then we have a fine  moduli scheme $\X_s := \M_G(\calC_s)$ of 
$G$-torsors on the smooth curve $\calC_s$ for each $s \in S$.
They form a flat family  $\X \to \calC$ over $S$.
Let us write $\pi:\X \to S$ the projection.
%By \S\ref{subsec:hks:G:seq} we have the short exact sequence
%\[
% 0 \longto \Theta_{\Y/S} \longto \Theta_{\pi} 
%   \longto \pi^{-1}\Theta_S \longto 0
%\]
%of sheaves of vector fields.

Set $\shA := (\pi^{-1} \Theta_{S})^G$,
which is the sheaf of $G$-invariant vector fields on the pull-back.
It is a Lie algebroid over $S$ with the anchor 
$\tau: \shA \to \Theta_{S}$.
%which make $\shA_{P}$ a Lie algebroid over $X$.
%If $H^0(X,\shA_P)=0$ and 
%there is a universal deformation of $(X,P)$.
%By Theorem \ref{thm:G-bundle}, we have a universal family 
By \S\ref{subsec:hks:G:seq} we have the short exact sequence 
\[
  0 \longto \shA_{\X/S}=\tau^{-1}\Theta_{\X/S} \longto 
  \shA_{\pi} = \tau^{-1}\Theta_{\pi}
  \longto \Theta_S \longto 0.
\]
The Lie algebroid $\Theta_{\pi}$ naturally acts on $\shA_{\pi}$.
Note that the Lie algebra $\Theta_{\X/S}$ acts via the $\Theta_{\pi}$-action.

Consider the data
\begin{enumerate}[\quad \ \ (a)]
\item
The (non dg) Lie algebroid $\shL=\Theta_S$ on $S$.
\setcounter{enumi}{1}
\item
The Lie $\pi^{-1}\shO_{S}$-algebra $L = \Theta_{\X/S}$.
\item
The Lie $\pi^{-1}\shO_S$-algebroid extension $\shK$ 
of $\pi^{\sharp}\shL=\Theta_{\pi}$ by $L=\Theta_{\X/S}$ 
sitting in the exact commutative diagram
\[
 \xymatrix{
    0                \ar[r]
  & \Theta_{\X/S}    \ar[r] \ar[d]_{\Delta}
  & \Theta_{\pi}     \ar[r] \ar[d]
  & \pi^{-1}\Theta_S \ar[r] \ar@{=}[d]
  & 0
  \\
    0             \ar[r]
  & \shK_0        \ar[r]_{j} 
  & \shK          \ar[r] 
  & \pi^{-1}\shL  \ar[r]
  & 0
 }.
\] 
Here 
$\Delta:\Theta_{\X/S} \inj \shK_0 
 := \Theta_{\X/S} \otimes L = \Theta_{\X/S}^{\otimes 2}$
is the diagonal embedding 
\item
The section $s:\Theta_{\X/S} \to \shK$ given by the first component of 
the morphism $j$ in (c).
\item
The left $\shK$-module structure on $L=\Theta_{\X/S}$ 
naturally arising from $\Theta_{\X/S} \subset \shK$ and the adjoint action,
\item[\quad \ \ (f')] 
The $\shO_S$-flat Lie algebroid $\shA_{\pi}$ on $\X/S$.
\setcounter{enumi}{6}
\item
The left $\shK$-module structure on $\shA_{\pi}$ arising from 
the  $\Theta_{\pi}$-action.
\item[\quad \ \ (h')]
The  morphism $\oln{\iota}: L=\Theta_{\X/S} \to \shA_{\pi}/\shO_{\X/S}$ of 
Lie algebras on $\X/S$ arising the action of 
$\Theta_{\X/S} \subset \Theta_{\pi}$ on $\shA_{\pi}$.
\end{enumerate}

The conditions (1)--(6) hold, so there is 
a left $\Theta_S^{\flat}$-module structure on 
$R^i \pi_*(\R(\X/S),J(\shA_{\pi}))$ for each $i$.
Setting $i=0$ and noting that a $\Theta_S$-module is 
equivalent to a projective $\Diff_S$-module structure 
by Example \ref{eg:Lalgd:Theta}.
Then seen at the  second order structure,
there is a projective flat connection on the space of relative global section.
The first order part of $\pi_*(\R(\X/S),J(\shA_{\pi}))$ 
is the vector bundle on $S$ whose fiber over $s \in S$ is 
the space $\Gamma(\X_s,\det)$ of global sections of 
the determinant line bundle $\det$ on the fine moduli scheme $\X_s$
of $G$-torsors.
In other words, 
$\Gamma(\X_s,\det)$ is the space of generalized theta functions.
Consequently  we recover Hitchin's result \cite{H,R2}.

\begin{thm}
There is a projective flat connection on the vector bundle on $S$ 
with fiber the space of generalized theta functions.
\end{thm}


\begin{thebibliography}{MMM}

\bibitem[BD04]{BD}
Beilinson, A., Drinfeld, V.,
\emph{Chiral algebras}, 
American Mathematical Society Colloquium Publications, \textbf{51}, 
American Mathematical Society, Providence, RI, 2004.
%17Bxx (14F43) 

\bibitem[BG92]{BG}
Beilinson,~A., Ginzburg,~V.,
\emph{Infinitesimal structure of moduli spaces of $G$-bundles}, 
Internat.\ Math.\ Res.\ Notices, \textbf{1992} (1992), no.\ 4, 63--74. 
%14H60 (32G08 32L05) 

\bibitem[BF04]{BZF}
Ben-Zvi,~D., Frenkel,~E.,
\emph{Geometric realization of the Segal-Sugawara construction},
in \emph{Topology, geometry and quantum field theory}, 46--97, 
London Math.\ Soc.\ Lecture Note Ser., \textbf{308}, 
Cambridge Univ. Press, Cambridge, 2004. 
%17B67 (17B68 17B69) 


\bibitem[EV94]{EV}
Esnault,~H.\ Viehweg,~E.\ 
\emph{Higher Kodaira-Spencer classes}. 
Math.\ Ann., \textbf{299} (1994), no.\ 3, 491--527. 
%14D15 (13D10 32G99) 

\bibitem[FB04]{FBZ}
Frenkel, E., Ben-Zvi, D.,
\emph{Vertex algebras and algebraic curves}, second edition, 
Mathematical Surveys and Monographs, \textbf{88}. 
American Mathematical Society, Providence, RI, 2004.
%17B69 (14D21 14H81) 

\bibitem[FM07]{FM}
Fiorenza,~D., Manetti, ~M.,
\emph{$L_\infty$ structures on mapping cones}, 
Algebra Number Theory, \textbf{1}, (2007), 301--330.

\bibitem[G95]{G}
Ginzburg,~V.,
\emph{Resolution of diagonals and moduli spaces}, 
in \emph{The moduli space of curves (Texel Island, 1994)}, 231--266, 
Progr.\ Math., \textbf{129}, Birkh\"{a}user Boston, Boston, MA, 1995. 
%14D20 (14B10 14H10) 

\bibitem[HS97]{HS}
Hinich,~V., Schechtman,~V.,
\emph{Deformation theory and Lie algebra homology. I, II}, 
Algebra Colloq., \textbf{4} (1997), no.\ 2, 213--240;
no.\ 3, 291--316.
%14B12 (17B56) 

%\bibitem[HS97b]{HS2}
%Hinich,~V.,Schechtman,~V.,
%\emph{Deformation theory and Lie algebra homology. II}, 
%Algebra Colloq.\ \textbf{4} (1997), no.\ 3, 291--316.


\bibitem[H90]{H}
Hitchin,~N.~J.,
\emph{Flat connections and geometric quantisation}, 
Commun.\ Math.\ Phys., \textbf{131} (1990), 347--380.

\bibitem[KS58]{KS}
Kodaira,~K., Spencer,~D.~C.,
\emph{On deformations of complex analytic structures, I, II, III},
Ann.\ of Math., \textbf{67} (1958), 328--466;
\textbf{71} (1960), 43--76.

\bibitem[LV12]{LV}
Loday,~J., Vallette,~B., 
\emph{Algebraic operads}, 
Grundlehren der Mathematischen Wissenschaften, 
\textbf{346}, Springer, Heidelberg, 2012.


\bibitem[Q69]{Q}
Quillen,~D.,
\emph{Rational homotopy theory},
Ann.\ of Math., (2) \textbf{90} (1969), 205--295.

\bibitem[R93]{R0}
Ran,~Z.,
\emph{Derivatives of Moduli}, 
Internat.\ Math.\ Res.\ Notices, \textbf{1993} (1993), 93--106.

\bibitem[R00]{R1}
Ran,~Z.,
\emph{Canonical infinitesimal deformations}, 
J.\ Algebraic Geom., 9 (2000), no.\ 1, 43--69.
%14D15

\bibitem[R06]{R2}
Ran,~Z.,
\emph{Jacobi cohomology, local geometry of moduli spaces, and Hitchin connections},
Proc.\ London Math.\ Soc., (3) \textbf{92} (2006), 545--580.
%14D15 (32G05) 


\bibitem[R08]{R3}
Ran,~Z.,
\emph{Lie atoms and their deformations},
GAFA, Geom.\ funct.\ anal., \textbf{18} (2008), 184--221.


%\bibitem[TUY89]{TUY}
%Tsuchiya, A., Ueno, K., Yamada, Y., 
%\emph{Conformal field theory on universal family of stable curves with gauge symmetries}, in  \emph{Integrable systems in quantum field theory and statistical mechanics}, 459--566, 
%Adv.\ Stud.\ Pure Math., \textbf{19}, Academic Press, Boston, MA, 1989. 
%81T40 (14H15 17B67 17B81 32G15) 

\end{thebibliography}
\end{document}